\documentclass{article}
\usepackage{amssymb,amsmath,bm}
\usepackage{amsthm}
\usepackage{hyperref}
\usepackage{mathrsfs}
\usepackage{textcomp}
\usepackage{enumerate}
\usepackage{graphicx}
\usepackage{caption}

\newtheorem{theorem}{Theorem}[section]
\newtheorem{lemma}[theorem]{Lemma}
\newtheorem{proposition}[theorem]{Proposition}

\newtheorem{corollary}[theorem]{Corollary}
\newtheorem{conjecture}[theorem]{Conjecture}

\theoremstyle{remark}
\newtheorem{remark}[theorem]{Remark}

\theoremstyle{remark}

\numberwithin{equation}{section}

\begin{document}

\title{\bf On type-preserving representations of the four-punctured sphere group}

\author{Tian Yang}

\date{}

\maketitle

\begin{abstract}
We give counterexamples to a question of Bowditch that if a non-elementary type-preserving representation $\rho:\pi_1(\Sigma_{g,n})\rightarrow PSL(2;\mathbb R)$ of a punctured surface group sends every non-peripheral simple closed curve to a hyperbolic element, then must $\rho$ be Fuchsian. The counterexamples come from relative Euler class $\pm1$ representations of the four-punctured sphere group. We also show that the mapping class group action on each non-extremal component of the character space of type-preserving representations of the four-punctured sphere group is ergodic, which confirms a conjecture of Goldman for this case. The main tool we use is Kashaev-Penner's lengths coordinates of the decorated character spaces.
\end{abstract}


\section{Introduction}

Let $\Sigma_g$ be an oriented  closed surface of genus $g\geqslant 2.$ The $PSL(2,\mathbb R)$-representation space $\mathcal R(\Sigma_g)$  is the space of group homomorphisms $\rho:\pi_1(\Sigma_g)\rightarrow PSL(2,\mathbb R)$ from the fundamental group of $\Sigma_g$ into $PSL(2,\mathbb R),$ endowed with the compact open topology.  The Euler class $e(\rho)$ of $\rho$ is the Euler class of the associated $S^1$-bundle on $\Sigma_g,$ which satisfies the Milnor-Wood inequality
$2-2g\leqslant e(\rho) \leqslant 2g-2.$
In \cite{G1}, Goldman proved that equality holds if and only if $\rho$ is Fuchsian, i.e., discrete and faithful; and in \cite{G2}, he proved that the connected components of $\mathcal R(\Sigma_g)$ are indexed by the Euler classes. I.e., for each integer $k$ with $|k|\leqslant 2g-2,$ the representations of Euler class $k$ exist and form a connected component of $\mathcal R(\Sigma_g).$ The Lie group $PSL(2,\mathbb R)$ acts on $\mathcal R(\Sigma_g)$ by conjugation, and the quotient space
$$\mathcal M(\Sigma_g)=\mathcal R(\Sigma_g)/PSL(2,\mathbb R)$$
 is the character space of $\Sigma_g.$ Since the Euler classes are preserved by conjugation, the connected components of $\mathcal M(\Sigma_g)$ are also indexed by the Euler classes, i.e.,
 $$\mathcal M(\Sigma_g)=\coprod_{k=2-2g}^{2g-2}\mathcal M_k(\Sigma_g),$$
 where $\mathcal M_k(\Sigma_g)$ is the space of conjugacy classes of representations of Euler class $k$. By the results of Goldman\,\cite{G1,G2}, the extremal components $\mathcal M_{\pm(2-2g)}(\Sigma_g)$ are respectively identified with the Teichm\"uller space of $\Sigma_g$ and that of $\Sigma_g$ endowed with the opposite orientation.

The mapping class group $Mod(\Sigma_g)$ of $\Sigma_g$ is the group of isotopy classes of orientation preserving self-diffeomorphisms of $\Sigma_g.$ By the Dehn-Nielsen Theorem, $Mod(\Sigma_g)$ is naturally isomorphic to the group of positive outer-automorphisms $Out^+(\pi_1(\Sigma_g)),$ which acts on $\mathcal M(\Sigma_g)$ preserving the Euler classes. Therefore, $Mod(\Sigma_g)$ acts on each connected component of $\mathcal M(\Sigma_g).$ It is well known (see e.g. Fricke\,\cite{Fr}) that the $Mod(\Sigma_g)$-action is properly discontinuous on the extremal components $\mathcal M_{\pm(2-2g)}(\Sigma_g),$ i.e., the Teichm\"uller spaces, and the quotients are the Riemann moduli spaces of complex structures on $\Sigma_g.$ On the non-extremal components $\mathcal M_k(\Sigma_g),$ $|k|<2g-2,$ Goldman conjectured  in \cite{G3} that the $Mod(\Sigma_g)$-action is ergodic with respect to the measure induced by the Goldman symplectic form \cite{G4}.

 Closely related to Goldman's conjecture is a question of Bowditch\,\cite{Bo}, Question C, that whether for each non-elementary and non-extremal (i.e. non-Fuchsian) representation $\rho$ in $\mathcal R(\Sigma_g),$ there exists a simple closed curve $\gamma$ on $\Sigma_g$ such that $\rho([\gamma])$ is an elliptic or a parabolic element of $PSL(2,\mathbb R).$ Recall that a representation is non-elementary if its image is Zariski-dense in $PSL(2,\mathbb R).$ Recently, March\'e-Wolff\,\cite{MW} show that an affirmative answer to Bowditch's question implies that Goldman's conjecture is true. More precisely, they show that for $(g,k)\neq (2,0),$ $Mod(\Sigma_g)$ acts ergodically on the subset of $\mathcal M_k(\Sigma_g)$ consisting of representations that send some simple closed curve on $\Sigma_g$ to an elliptic or parabolic element. Therefore, when $(g,k)\neq (2,0),$ the $Mod(\Sigma_g)$-action on $\mathcal M_k(\Sigma_g)$ is ergodic if and only if the subset above has full measure in $\mathcal M_k(\Sigma_g).$ In the same work, they answer Bowditch's question affirmatively for the genus $2$ surface $\Sigma_2,$ implying the ergodicity of the $Mod(\Sigma_2)$-action on $\mathcal M_{\pm1}(\Sigma_2).$ For the action on  the component $\mathcal M_0(\Sigma_2),$ they find two $Mod(\Sigma_2)$-invariant open subsets due to the existence of the hyper-elliptic involution, and show that on each of them the $Mod(\Sigma_2)$-action is ergodic. For higher genus surfaces $\Sigma_g,$ $g\geqslant 3,$ Souto \cite{S} recently gives an affirmative answer to Bowditch's question for the Euler class $0$ representations, proving the ergodicity of the $Mod(\Sigma_g)$-action on $\mathcal M_0(\Sigma_g).$ For $g\geqslant3$ and $k\neq 0,$ both Bowditch's question and Goldman's conjecture are still open.
 \\

Bowditch's question was originally asked for the type-preserving representations of punctured surface groups. Recall that a punctured surface $\Sigma_{g,n}$ of genus $g$ with $n$ punctures is a closed surface $\Sigma_g$ with $n$ points removed. Through out this paper, we required that the Euler characteristic of $\Sigma_{g,n}$ is negative. A \emph{peripheral element} of $\pi_1(\Sigma_{g,n})$ is an element that is represented by a curve freely homotopic to a circle that goes around a single puncture of $\Sigma_{g,n}.$ A representation $\rho:\pi_1(\Sigma_{g,n})\rightarrow PSL(2,\mathbb R)$ is called \emph{type-preserving} if it sends every peripheral element of $\pi_1(\Sigma_{g,n})$ to a parabolic element of $PSL(2,\mathbb R).$ In \cite{Bo}, Question C asks whether it is true that if a non-elementary type-preserving representation of a punctured surface group sends every non-peripheral simple closed curve to a hyperbolic element, then $\rho$ must be Fuchsian.

The main result of this paper gives counterexamples to this question. To state the result, we recall that there is the notion of relative Euler class $e(\rho)$ of a type-preserving representation $\rho$ that satisfies the Milnor-Wood inequality
$$|e(\rho)|\leqslant2g-2+n,$$
and equality holds if and only if $\rho$ is Fuchsian (see \cite{G1, G2} and also Proposition \ref{MF}).

\begin{theorem}\label{main1} There are uncountably many non-elementary type-preserving representations $\rho:\pi_1(\Sigma_{0,4})\rightarrow PSL(2,\mathbb R)$ with relative Euler class $e(\rho)=\pm1$ that send every non-peripheral simple closed curve to a hyperbolic element. In particular, these representations are not Fuchsian.\end{theorem}

Our method is to use Penner's lengths coordinates for the decorated character space defined by Kashaev in \cite{Ka1}. Briefly speaking, decorated character space of a punctured surface is the space of conjugacy classes of \emph{decorated representations}, namely, non-elementary type-preserving representations together with an assignment of \emph{horocycles} to the punctures. The lengths coordinate of a decorated representation depend on a choice of an ideal triangulation of the surface, and consists of the \emph{$\lambda$-lengths} of the edges determined by the horocycles, and of the \emph{signs} of the ideal triangles determined by the representation. The decorated Teichm\"uller space is a connected component of the decorated character space, and the restriction of the lengths coordinates to this component coincide with Penner's lengths coordinates. (See \cite{Ka1, Ka2} or Section \ref{Sec:2} for more details.) A key ingredient in the proof is Formula (\ref{trace}) of the traces of closed curves in the lengths coordinates, found in \cite{SY} by Sun and the author.  With a careful choice of an ideal triangulation of the four-punctured sphere, called a \emph{tetrahedral triangulation}, we show that the traces of three \emph{distinguished simple closed curves} are greater than $2$ in the absolute value if and only if the $\lambda$-lengths of edges in this triangulation satisfy certain anti-triangular inequalities. We then show that each simple closed curve is distinguished in some tetrahedral triangulation, and all tetrahedral triangulations are related by a sequence of moves, called the \emph{simultaneous diagonal switches}. By the change of $\lambda$-lengths formula (Proposition \ref{change}), we show that the anti-triangular inequalities are preserved by the simultaneous diagonal switches. Therefore, if the three distinguished simple closed curves are hyperbolic, then all the simple closed curves are hyperbolic. Finally, we show that there are uncountably many choices of the $\lambda$-lengths that satisfy the anti-triangular inequalities.

A consequence of Formula (\ref{trace}) is Theorem \ref{dominance} that each non-Fuchsian type-preserving representation is dominated by a Fuchsian one, in the sense that the traces of the simple closed curves of the former representation are less than or equal to those of the later in the absolute value. This is a counterpart of the result of Gueritaud-Kassel-Wolff\,\cite{GKW} and Deroin-Tholozan\,\cite{DT}, where they consider dominance of closed surface group representations.

Using the same technique, we also give an affirmative answer to Bowditch's question for the relative Euler class $0$ type-preserving representations of the four-punctured sphere group.

\begin{theorem}\label{main2} Every non-elementary type-preserving representation $\rho:\pi_1(\Sigma_{0,4})\rightarrow PSL(2,\mathbb R)$ with relative Euler class $e(\rho)=0$ sends some non-peripheral simple closed curve to an elliptic or parabolic element.
\end{theorem}

In contrast with the connected components of the character space of a closed surface, those of a punctured surface are more subtle to describe. For $\Sigma_{g,n}$ with $n\neq 0,$ denote by $\mathcal M_k(\Sigma_{g,n})$ be the space of conjugacy classes of type-preserving representations  with relative Euler class $k.$ As explained in \cite{Ka1},  $\mathcal M_k(\Sigma_{g,n})$ can be either empty or non-connected. For example, $\mathcal M_0(\Sigma_{0,3})=\mathcal M_0(\Sigma_{1,1})=\emptyset.$ The non-connectedness of  $\mathcal M_k(\Sigma_{g,n})$ comes from the existence  of two distinct conjugacy classes of parabolic elements of $PSL(2,\mathbb R).$ More precisely, each parabolic element of $PSL(2,\mathbb R)$ is up to $\pm I$ conjugate to an upper triangular matrix with trace $2,$ and its conjugacy class is distinguished by whether the sign of the non-zero off diagonal element is positive or negative.  Therefore, two type-preserving representations of the same relative Euler class which respectively send the same peripheral element into different conjugacy classes of parabolic elements  cannot be in the same connected components. Throughout this paper, we respectively call the two conjugacy class of parabolic elements the \emph{positive} and the \emph{negative} conjugacy classes. For a type-preserving representation $\rho:\pi_1(\Sigma_{g,n})\rightarrow PSL(2,\mathbb R),$ we say that the \emph{sign} of a puncture $v$ is \emph{positive}, denoted by $s(v)=1,$ if $\rho$ sends a peripheral element around this puncture into the positive conjugacy class of parabolic elements, and is \emph{negative}, denoted by $s(v)=-1,$ if otherwise. For an $s\in\{\pm1\}^n,$ we denote by  $\mathcal M_k^s(\Sigma_{g,n})$ the space of conjugacy classes of type-preserving representations with relative Euler class $k$ and signs of the punctures $s.$ It is conjectured in \cite{Ka1} that each $\mathcal M_k^s(\Sigma_{g,n}),$ if non-empty, is connected. The result confirms this for the four-punctured sphere.

\begin{theorem}\label{main3} Let $s\in\{\pm 1\}^4.$ Then
\begin{enumerate}[(1)]
\item $\mathcal M_0^s(\Sigma_{0,4})$ is non-empty if and only if $s$ contains exactly two $-1$ and two $1,$

\item $\mathcal M_{1}^s(\Sigma_{0,4})$  is non-empty if and only if $s$ contains at most one $-1,$

\item $\mathcal M_{-1}^s(\Sigma_{0,4})$ is non-empty if and only if $s$ contains at most one $1,$ and

\item all the non-empty spaces above are connected.
\end{enumerate}
\end{theorem}

As a consequence of Theorem \ref{main3}, $\mathcal M_0(\Sigma_{0,4})$ has six connected components and each of  $\mathcal M_{\pm1}(\Sigma_{0,4})$ has five connected components. The main tool we use in the proof is still the lengths coordinates; and we hope that the technique could be used for the other punctured surfaces. 

The mapping class group $Mod(\Sigma_{g,n})$ of a punctured surface $\Sigma_{g,n}$ is the group of relative isotopy classes of orientation preserving self-diffeomorphisms of $\Sigma_{g,n}$ that fix the punctures. By the Dehn-Nielsen Theorem, $Mod(\Sigma_{g,n})$ is isomorphic to the group of positive outer-automorphisms $Out^+(\pi_1(\Sigma_{g,n}))$ that preserve the cyclic subgroups  of $\pi_1(\Sigma_{g,n})$ generated by the peripheral elements, and hence acts on $\mathcal M(\Sigma_{g,n})$ preserving the relative Euler classes and the signs of the punctures. Therefore, for any integer $k$ with $|k|\leqslant 2g-2+n$ and for any $s\in\{\pm 1\}^n,$ $Mod(\Sigma_{g,n})$ acts on  $\mathcal M_k^s(\Sigma_{g,n}).$ For the four-punctured sphere, we have

\begin{theorem}\label{main4} The $Mod(\Sigma_{0,4})$-action on each non-extremal connected component of $\mathcal M(\Sigma_{0,4})$ is ergodic.
\end{theorem}

By March\'e-Wolff\,\cite{MW}, it is not surprising that the $Mod(\Sigma_{0,4})$-action is ergodic on the connected components of $\mathcal M(\Sigma_{0,4})$ where Bowditch's question has an affirmative answer. A new and unexpected phenomenon Theorem \ref{main4} reveals here is that, for punctured surfaces, the action of the mapping class group can still be ergodic when the answer to Bowditch's question is negative. Evidenced by Theorem \ref{main4}, we make the following

\begin{conjecture} The $Mod(\Sigma_{g,n})$-action is ergodic on each non-extremal connected component of $\mathcal M (\Sigma_{g,n}).$
\end{conjecture}

The paper is organized as follows. In Section \ref{Sec:2}, we recall Kashaev's decorated character spaces and the lengths coordinates, in Section \ref{Sec:3}, we obtain a formula of the traces of closed curves in the lengths coordinates, and in Section \ref{Sec:4}, we introduce tetrahedral triangulations, distinguished simple closed curves and simultaneous diagonal switches. Then we prove Theorems \ref{main1} and \ref{main2}, Theorem \ref{main3} and Theorem\ref{main4} respectively in Sections \ref{Sec:5}, Section \ref{Sec:6} and Section \ref{Sec:7}. It is pointed out by the anonymous referee that the results concerning representations of relative Euler class $\pm1$ can be deduced more directly from the results of Goldman in \cite{G5}, where  we present his argument in Appendix \ref{B} for the interested readers.
\\

\textbf{Acknowledgments:} The author is grateful to the referee for bringing his attention to the relationship of this work with a previous work of Goldman, and for several valuable suggestions to improve the writing of this paper. The author also would like to thank Steven Kerckhoff, Feng Luo and Maryam Mirzakhani for helpful discussions and suggestions, Ser Peow Tan, Maxime Wolff, Sara Moloni, Fr\'ed\'eric Palesi and Zhe Sun for discussion and showing interest, Julien Roger for bringing his attention to Kashaev's work on the decorated character variety, Ronggang Shi for answering his questions on ergodic theory and Yang Zhou for writing a Python program for testing some of the author's ideas.
\medskip

The author is supported by NSF grant DMS-1405066.

\section{Decorated character spaces}\label{Sec:2}

We recall the decorated character spaces and the lengths coordinates in this section. The readers are recommended to read Kashaev's papers \cite{Ka1,Ka2} for the original approach and for more details.

Let $\Sigma_{g,n}$ be a punctured surface of genus $g$ with $n$ punctures, and let $\rho:\pi_1(\Sigma_{g,n})\rightarrow PSL(2,\mathbb R)$ be a non-elementary type-preserving representation.  A \emph{pseudo-developing map} $D_{\rho}$ of $\rho$ is a piecewise smooth $\rho$-equivariant map from the universal cover of $\Sigma_{g,n}$ to the hyperbolic plane $\mathbb H^2.$ By \cite{G1}, $\rho$ is the holonomy representation of $D_{\rho}.$  Let $\omega$ be the hyperbolic area form of $\mathbb H^2.$ Since $D_{\rho}$ is $\rho$-equivariant, the pull-back $2$-form $(D_{\rho})^*\omega$ descends to $\Sigma_{g,n}.$ The relative Euler class $e(\rho)$ of $\rho$ could be calculated by
$$e(\rho)=\frac{1}{2\pi}\int_{\Sigma_{g,n}}(D_{\rho})^*\omega.$$

An ideal arc $\alpha$ on $\Sigma_{g,n}$ is an arc connecting two (possibly the same) punctures. The image $D_{\rho}(\tilde\alpha)$ of a lift $\tilde\alpha$ of $\alpha$ is an arc in $\mathbb H^2$ connecting two (possibly the same) points on $\partial \mathbb H^2,$ each of which is the fixed point of the $\rho$-image of certain peripheral element of $\pi_1(\Sigma_{g,n}).$ We call $\alpha$  \emph{$\rho$-admissible} if the two end points of $D_{\rho}(\tilde\alpha)$ are distinct.  It is easy to see that $\alpha$ being $\rho$-admissible is independent of the choice of the lift $\tilde\alpha$ and the pseudo-developing map $D_{\rho}.$ An ideal triangulation $\mathcal T$ of $\Sigma_{g,n}$ consists of a set of disjoint ideal arcs, called the edges, whose complement is a disjoint union of triangles, called the ideal triangles. We call $\mathcal T$  \emph{$\rho$-admissible} if all the edges of $\mathcal T$ are $\rho$-admissible. If $\rho'$ is conjugate to $\rho,$ then it is easy to see that $\mathcal T$ is $\rho'$-admissible if and only if it is $\rho$-admissible. In \cite{Ka1}, Kashaev shows that

\begin{theorem}[Kashaev]\label{Kashaev1}
For each ideal triangulation $\mathcal T,$ the set $$\mathcal M_{\mathcal T}(\Sigma_{g,n})=\big\{[\rho]\in\mathcal M(\Sigma_{g,n})\ \big|\ \mathcal T\text{ is } \rho\text{-admissible}\big\}$$ is open and dense in $\mathcal M(\Sigma_{g,n}),$ and there exist finitely many ideal triangulations $\mathcal T_i,$ $i=1,\dots,m,$ such that
$$\mathcal M(\Sigma_{g,n})=\bigcup_{i=1}^m \mathcal M_{\mathcal T_i}(\Sigma_{g,n}).$$
\end{theorem}

Let $\Sigma_{g,n}$ and $\rho$ be as above. A \emph{decoration} of $\rho$ is an assignment of horocycles centered at the fixed points of the $\rho$-image of the peripheral elements of $\pi(\Sigma_{g,n}),$ one for each, which is invariant under the $\rho(\pi_1(\Sigma_{g,n}))$-action.  In the case that the fixed points of the $\rho$-image of two peripheral elements coincide, which may happen only when $\rho$ is non-Fuchsian, we do not require the corresponding assigned horocycles to be the same. If $d$ is a decoration of $\rho$ and $g$ is an element of $PSL(2,\mathbb R),$ then the $g$-image of the horocycles in $d$ form a decoration $g\cdot d$ of the conjugation $g\rho g^{-1}$ of $\rho.$ We call a pair $(\rho,d)$ a \emph{decorated representation}, and call two decorated representations $(\rho,d)$ and $(\rho', d')$ \emph{equivalent} if $\rho'=g\rho g^{-1}$ and $d'=g\cdot  d$ for some $g\in PSL(2,\mathbb R).$ The \emph{decorated character space} of $\Sigma_{g,n},$ denoted by $\mathcal M^d(\Sigma_{g,n}),$ is the space of equivalence classes of decorated representations. In \cite{Ka1}, Kashaev shows that the projection $\pi:\mathcal M^d(\Sigma_{g,n})\rightarrow\mathcal M(\Sigma_{g,n})$ defined by $\pi([(\rho,d)])=[\rho]$ is a principle $\mathbb R^V_{>0}$-bundle, where $V$ is the set of punctures of $\Sigma_{g,n},$ and the pre-image of the extremal components are isomorphic as principle $\mathbb R^V_{>0}$-bundles to Penner's decorated Teichm\"uller space\,\cite{P1}.

Fixing a $\rho$-admissible ideal triangulation $\mathcal T$ and a pseudo-developing map $D_{\rho},$ the lengths coordinate of $(\rho,d)$ consists of the following two parts: the $\lambda$-lengths of the edges and the signs of the ideal triangles. The $\lambda$-length of an edge $e$ of $\mathcal T$ is defined as follows. Since $e$ is $\rho$-admissible, for any lift $\widetilde e$ of $e$ to the universal cover the image $D_{\rho}(\widetilde e)$ connects to distinct points $u_1$ and $u_2$ on $\partial\mathbb H^2.$ The decoration $d$ assigns two horocycles $H_1$ and $H_2$  respectively centered at $u_1$ and $u_2.$ Let $l(e)$ be the signed hyperbolic distance between the two horocycles, i.e., $l(e)>0$ if $H_1$ and $H_2$ are disjoint and $l(e)\leqslant 0$ if otherwise. Then the \emph{$\lambda$-length} of $e$ in the decorated representation $(\rho, d)$ is defined by
$$\lambda(e)=\exp\frac{l(e)}{2}.$$
The sign of an ideal triangle $t$ of $\mathcal T$ is defined as follows. Let $v_1,$ $v_2$ and $v_3$ be the vertices of $t$ so that the orientation on $t$ determined by the cyclic order $v_ 1\mapsto v_2\mapsto v_3\mapsto v_1$ coincides with the one induced from the orientation of $\Sigma_{g,n}.$ Let $\tilde t$ be a lift of $t$ to the universal cover, and let $\widetilde{v_1},$ $\widetilde{v_2}$ and $\widetilde{v_3}$ be the vertices of $\tilde t$ so that $\widetilde{v_i}$ is a lift of $v_i,$ $i=1,2,3.$ Since $\mathcal T$ is $\rho$-admissible, the points $D_{\rho}(\widetilde{v_1}),$ $D_{\rho}(\widetilde{v_2})$ and $D_{\rho}(\widetilde{v_3})$ are distinct on $\partial \mathbb H^2,$ and hence determine a hyperbolic ideal triangle $\Delta$ in $\mathbb H^2$ with them as the ideal vertices. The \emph{sign} of $t$ is positive, denoted by $\epsilon(t)=1,$ if the orientation on $\Delta$ determined by the cyclic oder $D_{\rho}(\widetilde{v_1})\mapsto D_{\rho}(\widetilde{v_2})\mapsto D_{\rho}(\widetilde{v_3})\mapsto D_{\rho}(\widetilde{v_1})$ coincides with the one induced from the orientation of $\mathbb H^2.$ Otherwise, the sign of $t$ is negative, and is denoted by $\epsilon(t)=-1.$ From the construction, it is easy to see that the $\lambda$-lengths $\lambda(e)$ and the signs $\epsilon(t)$ depend only on the equivalence class of $(\rho,d).$ Let $T$ be the set of ideal triangles of $\mathcal T.$ Then the integral of the pull-back form $(D_{\rho})^*\omega$ over $\Sigma_{g,n}$ equals $\sum_{t\in T}\epsilon(t)\pi,$ and the relative Euler class of $\rho$ can be calculated by
\begin{equation}\label{euler}
e(\rho)=\frac{1}{2}\sum_{t\in T}\epsilon(t).
\end{equation}

Let $V$ be the set of punctures of $\Sigma_{g,n}$ and let $E$ be the set of edges of $\mathcal T.$ Then there is a principle $\mathbb R_{>0}^V$-bundle structure on $\mathbb R_{>0}^E$ defined as follows. For $\mu\in\mathbb R_{>0}^V$ and $\lambda\in\mathbb R_{>0}^E,$ we define $\mu\cdot\lambda\in \mathbb R_{>0}^E$ by $(\mu\cdot\lambda)(e)=\mu(v_1)\lambda(e)\mu(v_2),$ where $v_1$ and $v_2$ are the punctures connected by the edge $e.$

\begin{theorem}[Kashaev]\label{Kashaev2}
Let $\pi:\mathcal M^d(\Sigma_{g,n})\rightarrow\mathcal M(\Sigma_{g,n})$ be the principle $\mathbb R_{>0}^V$-bundle, and let $\mathcal M_{\mathcal T}^d(\Sigma_{g,n})$ the pre-image of $\mathcal M_{\mathcal T}(\Sigma_{g,n}).$
Then $$\mathcal M_{\mathcal T}^d(\Sigma_{g,n})=\coprod_{\epsilon\in\{\pm1\}^T}\mathcal R(\mathcal T,\epsilon),$$
where each $\mathcal R(\mathcal T,\epsilon)$ is isomorphic as principle $\mathbb R_{>0}^V$-bundles to an open subset of $\mathbb R_{>0}^E.$ The isomorphism is given by the $\lambda$-lengths, and the image of $\mathcal R(\mathcal T,\epsilon)$ is the complement of the zeros of certain rational function coming from the image of the peripheral elements not being the identity matrix.
\end{theorem}

On $\mathcal M(\Sigma_{g,n})$ there is the Goldman symplectic form $\omega_{WP}$ which restricts to the Weil-Petersson symplectic form on the Teichm\"uller component\,\cite{G4}.
By \cite{Ka1,Ka2}, for each ideal triangulation $\mathcal T,$ the pull-back of $\omega_{WP}$ to $\mathcal M_{\mathcal T}^d(\Sigma_{g,n})$ is expressed in the $\lambda$-lengths by
\begin{equation}\label{WP}\pi^*\omega_{WP}=\sum_{t\in T}\Big(\frac{d\lambda(e_1)\wedge d\lambda(e_2)}{\lambda(e_1)\lambda(e_2)}+\frac{d\lambda(e_2)\wedge d\lambda(e_3)}{\lambda(e_2)\lambda(e_3)}+\frac{d\lambda(e_3)\wedge d\lambda(e_1)}{\lambda(e_3)\lambda(e_1)}\Big),
\end{equation}
where $e_1,$ $e_2$ and $e_3$ are the edges of the ideal triangle $t$ in the cyclic order induced from the orientation of $\Sigma_{g,n}.$ This formula is first obtained by Penner\,\cite{P2} for the decorated Teichm\"uller space. From (\ref{WP}), it is easy to see that the measure on each $\mathcal R(\mathcal T,\epsilon)$ induced by $\pi^*\omega_{WP}$ is in the measure class of the pull-back of the Lebesgue measure of $\mathbb R_{>0}^E.$

A diagonal switch at an edge $e$ of $\mathcal T$ replaces the edge $e$ by the other diagonal of the quadrilateral formed by the union of the two ideal triangles adjacent to $e.$ By \cite{Ha}, any ideal triangulation can be obtained from another by doing a finite sequence of diagonal switches. Let $(\rho,d)$ be a decorated representation and let $\mathcal T$ be a $\rho$-admissible ideal triangulation of $\Sigma_{g,n}.$ If $\mathcal T'$ is the ideal triangulation of $\Sigma_{g,n}$ obtained from $\mathcal T$ by doing a diagonal switch at an edge $e,$ then the $\rho$-admissibility of $\mathcal T'$ and the lengths coordinate of $(\rho,d)$ in $\mathcal T'$ are determined as follows. Let $t_1$ and $t_2$ be the two ideal triangles of $\mathcal T$ adjacent to $e,$ let $e'$ be the new edge of $\mathcal T'$ and let $t_1'$ and $t_2'$ be the two ideal triangles in $\mathcal T'$ adjacent to $e'.$ We respectively name the edges of the quadrilateral $e_1,\dots, e_4$ in the way that  $e_1$ is adjacent to $t_1$ and $t_1',$ $e_2$ is adjacent to $t_1$ and $t_2',$ $e_3$ is adjacent to  $t_2$ and $t_2'$ and $e_4$ is adjacent to $t_2$ and $t_1'.$ Then $e_1$ is opposite to $e_3,$ and $e_2$ is opposite to $e_4$ in the quadrilateral.

\begin{proposition}[Kashaev]\label{change} \begin{enumerate}[(1)]

\item If the signs $\epsilon(t_1)=\epsilon(t_2),$ then $\mathcal T'$ is $\rho$-admissible. In this case, $$\epsilon(t_1')=\epsilon(t_2')=\epsilon(t_1) \quad\text{and}\quad\lambda(e')=\frac{\lambda(e_1)\lambda(e_3)+\lambda(e_2)\lambda(e_4)}{\lambda(e)},$$
and the signs of the common ideal triangles and the $\lambda$-lengths of the common edges of $\mathcal T$ and $\mathcal T'$ do not change.

\item If $\epsilon(t_1)\neq\epsilon(t_2),$ then $\mathcal T'$ is $\rho$-admissible if and only if $\lambda(e_1)\lambda(e_3)\neq\lambda(e_2)\lambda(e_4).$ In this case,

\begin{enumerate}[(2.1)]
\item if $\lambda(e_1)\lambda(e_3)<\lambda(e_2)\lambda(e_4),$ then $$\epsilon(t_1')=\epsilon(t_1),\ \  \ \epsilon(t_2')=\epsilon(t_2)\quad\text{and}\quad\lambda(e')=\frac{\lambda(e_2)\lambda(e_4)-\lambda(e_1)\lambda(e_3)}{\lambda(e)},$$
\item if $\lambda(e_2)\lambda(e_4)<\lambda(e_1)\lambda(e_3),$ then $$\epsilon(t_1')=\epsilon(t_2),\ \ \ \epsilon(t_2')=\epsilon(t_1)\quad\text{and}\quad\lambda(e')=\frac{\lambda(e_1)\lambda(e_3)-\lambda(e_2)\lambda(e_4)}{\lambda(e)},$$
\end{enumerate}
and the signs of the common ideal triangles and the $\lambda$-lengths of the common edges of $\mathcal T$ and $\mathcal T'$ do not change, 

\end{enumerate}
\end{proposition}
The rule for the signs in (2.1) and (2.2) is that the signs of the ideal triangles adjacent to the shorter edges do not change. This could be seen as follows. If, for example, $\epsilon(t_1)=-1,$ $\epsilon(t_2)=1$ and $\lambda(e_1)\lambda(e_3)<\lambda(e_2)\lambda(e_4),$ then the hyperbolic ideal triangle $\Delta_1$ determined by $t_1$ is negatively oriented, the hyperbolic ideal triangle $\Delta_2$ determined by $t_2$ is positively oriented, and the geodesic arcs $a_2$ and $a_4$ determined by $e_2$ and $e_4$ intersect. See Figure \ref{Fig1}. As a consequence, the hyperbolic ideal triangle $\Delta_1'$ determined by $t_1'$ is negatively oriented and the hyperbolic ideal triangle $\Delta_2'$ determined by $t_2'$ is positively oriented, hence $\epsilon(t_1')=-1$ and $\epsilon(t_2')=1.$ The $\lambda$-lengths of $e'$ follows from Penner's Ptolemy relation\,\cite{P1} that $\lambda(e_2)\lambda(e_4)=\lambda(e_1)\lambda(e_3)+\lambda(e)\lambda(e').$ The other cases could be verified similarly. \begin{figure}[htbp]
\centering
\includegraphics[scale=0.35]{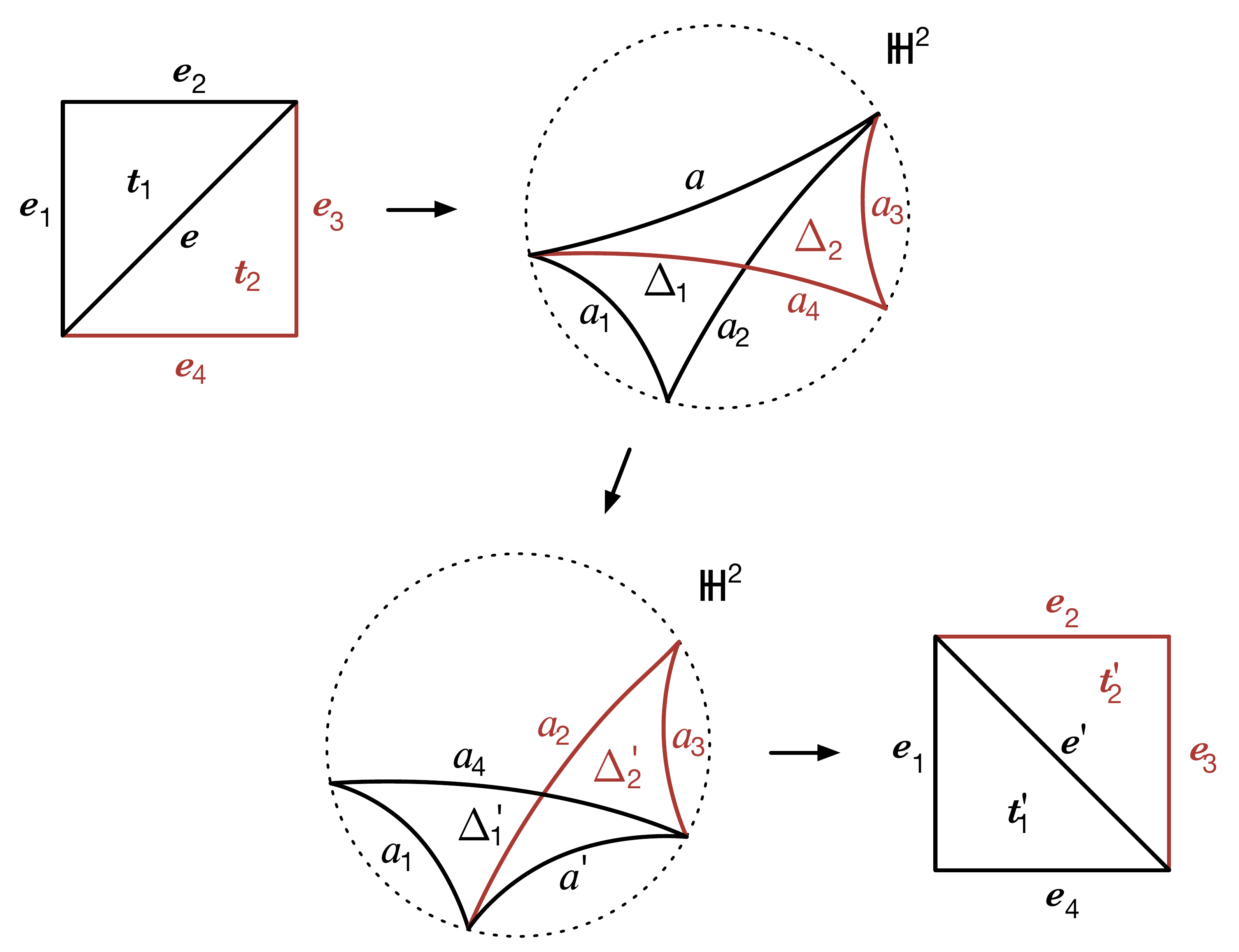}
\caption{}
\label{Fig1}
\end{figure}

By Theorems \ref{Kashaev1} and \ref{Kashaev2}, the family of open subsets $\{\mathcal R(\mathcal T,\epsilon)\}$ where $\mathcal T$ goes over all the ideal triangulations of $\Sigma_{g,n}$ together with the $\lambda$-lengths functions $\{\lambda:\mathcal R(\mathcal T,\epsilon)\rightarrow\mathbb R_{>0}^E\}$ form a coordinate system of $\mathcal M^d(\Sigma_{g,n}),$ and the transition functions are given by Proposition \ref{change}.

\begin{theorem}[Kashaev]\label{Kashaev3} For each relative Euler class $k,$
 let $\mathcal M_k^d(\Sigma_{g,n})$ be the pre-image of $\mathcal M_k(\Sigma_{g,n})$ under the projection $\pi:\mathcal M^d(\Sigma_{g,n})\rightarrow\mathcal M(\Sigma_{g,n}).$  Then
$$\mathcal M_k^d(\Sigma_{g,n})=\bigcup_{\mathcal T}\coprod_{\epsilon}\mathcal R(\mathcal T,\epsilon),$$
where the union is over all the ideal triangulations $\mathcal T$ and the disjoint union is over all $\epsilon\in\{\pm1\}^T$ such that $\sum_{t\in T}\epsilon(t)=2k.$ Moreover, $\mathcal M_k^d(\Sigma_{g,n})$'s are principle $\mathbb R_{>0}^V$-bundles, and are disjoint for different $k.$
\end{theorem}

\section{A trace formula for closed curves}\label{Sec:3}

Throughout this section, we let $\mathcal T$ be an ideal triangulation of $\Sigma_{g,n},$ and let $E$ and $T$ respectively be the set of edges and ideal triangles of $\mathcal T.$ Given the lengths coordinate $(\lambda,\epsilon)\in\mathbb R_{>0}^E\times\{\pm1\}^T,$ up to conjugation, the type-preserving representation $\rho:\pi_1(\Sigma_{g,n})\rightarrow PSL(2,\mathbb R)$ can be reconstructed up to conjugation as follows. Suppose $e$ is an edge of $\mathcal T,$ and $t_1$ and $t_2$ are the two ideal triangles adjacent to $e.$ Let $e_1$ and $e_2$ be the other two edges of $t_1$ and let $e_3$ and $e_4$ be the other two edges of $t_2$ so that the cyclic orders $e\mapsto e_1\mapsto e_2\mapsto e$ and $e\mapsto e_3\mapsto e_4\mapsto e$ coincide with the one induced from the orientation of $\Sigma_{g,n}.$ Define the quantity $X(e)\in\mathbb R_{>0}$ by
$$X(e)=\frac{\lambda(e_2)\lambda(e_4)}{\lambda(e_1)\lambda(e_3)}.$$
Note that if $\rho$ is discrete and faithful, then $X(e)$ is the shear parameter of the corresponding hyperbolic structure at $e.$ (See e.g. \cite{B2}.) It is well known that each immersed closed curve on $\Sigma_{g,n}$ is homotopic to a normal one that transversely intersects each ideal triangle in simple arcs that connect different edges of the triangle. Let $\gamma$ be an immersed oriented closed normal curve on $\Sigma_{g,n}.$  For each edge $e$ intersecting $\gamma,$ define
\begin{equation*}
    S(e)= \left[
      \begin{array}{cc}
        X(e)^{\frac{1}{2}} & 0 \\
        0 & X(e)^{-\frac{1}{2}}
      \end{array} \right].
  \end{equation*}
For each ideal triangle $t$ intersecting $\gamma,$ define \begin{equation*}
   R(t)= \left[
      \begin{array}{cc}
      1 & \epsilon(t) \\
        0 &1
      \end{array} \right]
  \end{equation*}
if $\gamma$ makes a left turn in $t$ (Figure \ref{Fig2} (a)), and define
\begin{equation*}
 R(t)= \left[
      \begin{array}{cc}
      1 & 0\\
       \epsilon(t)  &1
      \end{array} \right]
  \end{equation*}
if $\gamma$ makes a right turn in $t$ (Figure \ref{Fig2} (b)).

\begin{figure}[htbp]
\centering
\includegraphics[scale=0.35]{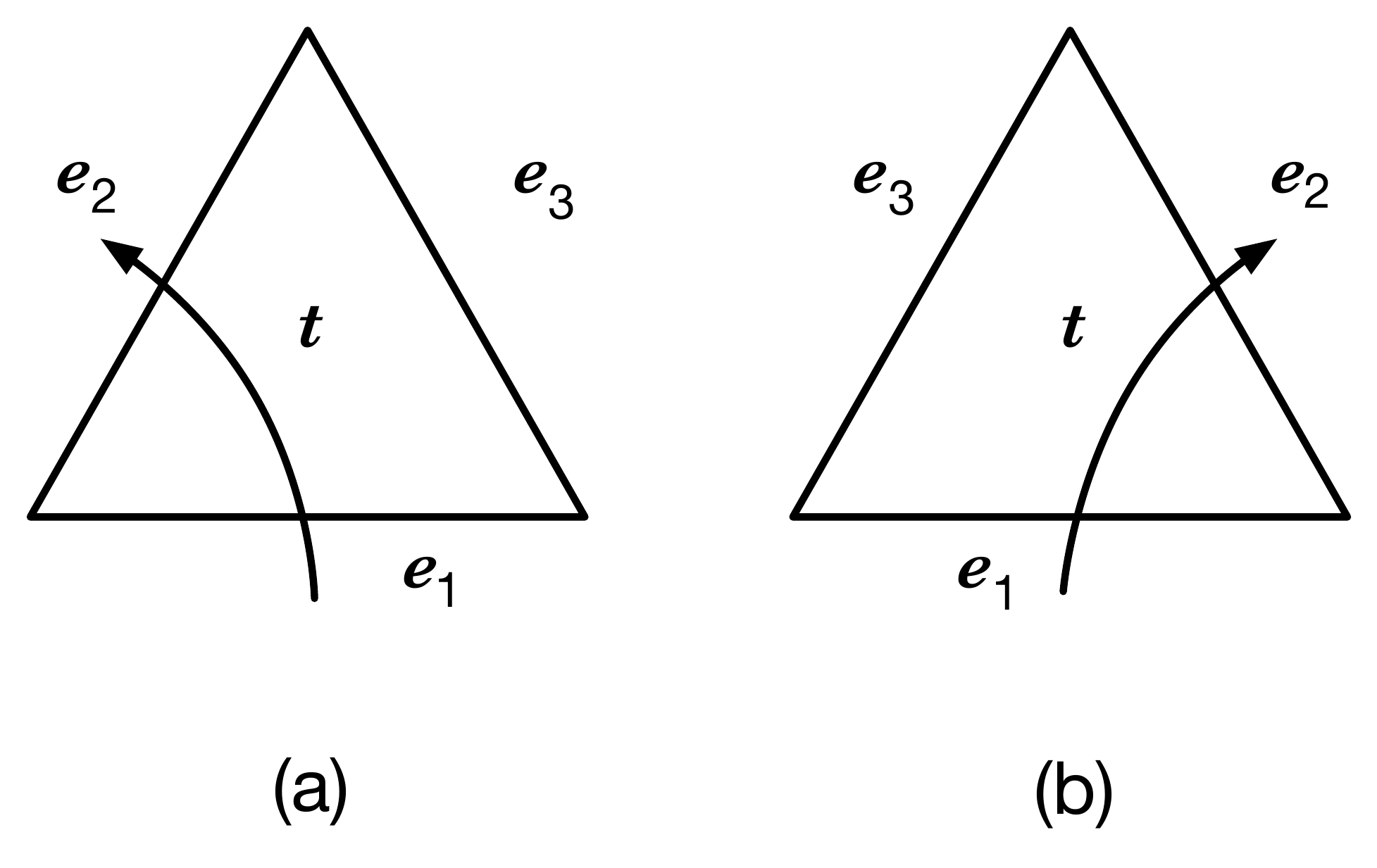}
\caption{}
\label{Fig2}
\end{figure}

\begin{lemma}\label{3.1} Let $e_{i_1},\dots,e_{i_m}$ be the edges and let $t_{j_1},\dots,t_{j_m}$ be the ideal triangles of $\mathcal T$ intersecting $\gamma$ in the cyclic order induced by the orientation of $\gamma$ so that $e_{i_k}$ is the common edge of $t_{j_{k-1}}$ and $t_{j_k}$ for each $k\in\{1,\dots,m\}.$ Then up to a conjugation by an element of $PSL(2,\mathbb R),$
$$\rho([\gamma])=\pm S(e_{i_1})R(t_{j_1})S(e_{i_2})R(t_{j_2})\dots S(e_{i_m})R(t_{j_m}).$$
\end{lemma}

\begin{proof} The proof is parallel to that of Lemma 3 in \cite{BW}. The idea is to keep track of the image of the unit tangent vector $\frac{\partial}{\partial y}$ at $i\in\mathbb H^2$ under $\rho([\gamma]).$ The contributions of each edge $e$ and of each ideal triangle $t$ intersecting $\gamma$  to $\rho([\gamma])$ are respectively $\pm S(e)$ and $\pm R(t).$ See also Exercise 8.5-8.7 and 10.14 in \cite{B}.
\end{proof}

For each puncture $v$ of $\Sigma_{g,n},$ let $\gamma_v$ be the simple closed curve going counterclockwise around $v$ once. By Lemma \ref{3.1},  the image of $\gamma_v$ is up to conjugation 
\begin{equation*}
 \rho([\gamma_v])= \pm\left[
      \begin{array}{cc}
      1 & \psi_{v,\epsilon}(\lambda)\\
    0&1
      \end{array} \right],
  \end{equation*}
where $\psi_{v,\epsilon}$ is a rational function of $\lambda$ depending on $\epsilon.$ Therefore, $\rho$ is type-preserving if and only if $\psi_{v,\epsilon}(\lambda)\neq0$ for all punctures $v.$ The following proposition gives a more precise description of this rational function in Theorem \ref{Kashaev2}.

\begin{proposition}[Kashaev]\label{rational} Let $(\lambda,\epsilon)\in\mathbb R_{>0}^E\times\{\pm1\}^T,$ let $V$ be the set of punctures of $\Sigma_{g,n}$ and let $\psi_{\epsilon}$ be the rational function defined by 
$$\psi_{\epsilon}=\prod_{v\in V}\psi_{v,\epsilon}.$$ Then $(\lambda,\epsilon)$ defines a type-preserving representation if and only if $\psi_{\epsilon}(\lambda)\neq0.$
\end{proposition}

The following theorem provides a more direct way to calculate the absolute values of the traces of closed curves using the $\lambda$-lengths, which is first found by Sun and the author in \cite{SY}. We include a proof here for the reader's convenience. For each ideal triangle $t$ intersecting $\gamma,$ let $e_{1}$ be the edge of $t$ at which $\gamma$ enters,  let $e_{2}$ be the edge of $t$ at which $\gamma$ leaves and let $e_{3}$ be the other edge of $t.$ See Figure \ref{Fig2}. Define
\begin{equation*}
  M(t)= \left[
      \begin{array}{cc}
      \lambda(e_{1}) & \epsilon(t)\lambda(e_{3}) \\
        0 &\lambda(e_{2})
      \end{array} \right]
  \end{equation*}
if $\gamma$ makes a left turn in $t$, and define
\begin{equation*}
 M(t)= \left[
      \begin{array}{cc}
      \lambda(e_{2})  & 0\\
       \epsilon(t)\lambda(e_{3})  &\lambda(e_{1})
      \end{array} \right]
  \end{equation*}
if $\gamma$ makes a right turn in $t.$

\begin{theorem}\label{trace formula} For an immersed closed normal curve $\gamma$ on $\Sigma_{g,n},$ let $e_{i_1},\dots,e_{i_m}$ be the edges and let $t_{j_1},\dots,t_{j_m}$ be the ideal triangles of $\mathcal T$ intersecting $\gamma$ in the cyclic order following the orientation of $\gamma$ so that $e_{i_k}$ is the common edge of $t_{j_{k-1}}$ and $t_{j_k}$ for each $k\in\{1,\dots,m\}.$  Then
\begin{equation}\label{trace}
\big|tr\rho([\gamma])\big|=\frac{\big|tr\big(M(t_{j_1})\dots M(t_{j_m})\big)\big|}{\lambda(e_{i_1})\dots\lambda(e_{i_m})}.
\end{equation}

\end{theorem}

\begin{proof} For each ideal triangle $t$ and an edge $e$ of $t,$ let $e'$ and $e''$ be the other two edges of $t$ so that the cyclic order $e\mapsto e'\mapsto e''\mapsto e$ coincides with the one induced by the orientation of $\Sigma_{g,n}.$ Define the matrix
\begin{equation*}
 S(t,e)= \left[
      \begin{array}{cc}
   \sqrt{\frac{\lambda(e'')}{\lambda(e')}}  & 0\\
  0  &    \sqrt{\frac{\lambda(e')}{\lambda(e'')}}
      \end{array} \right].
  \end{equation*}
Then
\begin{equation}\label{split}
S(e_{i_k})=S(t_{j_{k-1}},e_{i_k})S(t_{j_k},e_{i_k})
\end{equation}
 for each $k\in\{1,\dots,m\},$ where as a convention $t_{j_{1-1}}=t_{j_m}.$ A case by case calculation shows that
\begin{equation}\label{local}
S(t_{j_k},e_{i_k})R(t_{j_k})S(t_{j_k},e_{i_{k+1}})=\frac{M(t_{j_k})}{\sqrt{\lambda(e_{i_k})\lambda(e_{i_{k+1}})}}
\end{equation}
for each $k\in\{1,\dots,m\},$ where as a convention $e_{i_{m+1}}=e_{i_1}.$ By Lemma \ref{3.1}, (\ref{split}), (\ref{local}) and the fact that $tr(AB)=tr(BA)$ for any two matrices $A$ and $B,$ we have
\begin{equation*}
\begin{split}
\big|tr\rho([\gamma])\big|=&\big|tr\big(S(e_{i_1})R(t_{j_1})\dots S(e_{i_m})R(t_{j_m})\big)\big|\\
=&\big|tr\big(S(t_{j_1},e_{i_1})R(t_{j_1})S(t_{j_1},e_{i_2})\dots S(t_{j_m},e_{i_m})R(t_{j_m}) S(t_{j_m},e_{i_1})\big)\big|\\
=&\frac{\big|tr\big(M(t_{j_1})\dots M(t_{j_m})\big)\big|}{\lambda(e_{i_1})\dots\lambda(e_{i_m})}.
\end{split}
\end{equation*}
\end{proof}

\begin{remark} Formula (\ref{trace}) is first obtained by Roger-Yang\,\cite{RY} for decorated hyperbolic surfaces, i.e., discrete and faithful decorated representations, using the skein relations, where their formula includes both the traces of closed geodesics and the $\lambda$-lengths of geodesics arcs connecting the punctures. It is interesting to know whether there is a similar formula for the $\lambda$-lengths of arcs for the non-Fuchsian decorated representations. 
\end{remark}

As a consequence of Theorem \ref{trace formula}, we have the following

\begin{theorem}\label{dominance} 
\begin{enumerate}[(1)]
\item For every non-Fuchsian type-preserving representation $\rho:\Sigma_{g,n}\rightarrow PSL(2,\mathbb R),$ there exists a Fuchsian type-preserving representation $\rho'$ such that 
$$\big|tr\rho([\gamma])\big|\leqslant\big|tr\rho'([\gamma])\big|$$  
for each $[\gamma]\in\pi_1(\Sigma_{g,n}),$ and the strict equality holds for at least one $\gamma.$

\item Conversely, for almost every Fuchsian type-preserving representation $\rho':\Sigma_{g,n}\rightarrow PSL(2,\mathbb R)$ and for each $k$ with $|k|<2g-2+n$ and $\mathcal M_k(\Sigma_{g,n})\neq\emptyset,$ there exists a type-preserving representation $\rho$ with $e(\rho)=k$ such that
$$\big|tr\rho([\gamma])\big|\leqslant\big|tr\rho'([\gamma])\big|$$  
 for each $[\gamma]\in\pi_1(\Sigma_{g,n}),$ and the strict equality holds for at least one $\gamma.$
\end{enumerate} 
\end{theorem}

\begin{proof} For (1), by Theorem \ref{Kashaev1}, there exists a $\rho$-admissible ideal triangulation $\mathcal T.$ Choose arbitrarily a decoration $d$ of $\rho,$ and let $(\rho',d')$ be the decorated representation that has the same $\lambda$-lengths of $(\rho,d)$ and has the positive signs for all the ideal triangles. By (\ref{euler}) and Goldman's result in \cite{G1}, $\rho'$ is Fuchsian. Applying Formula (\ref{trace}) to $|tr\rho([\gamma])|$ and $|tr\rho'([\gamma])|,$ we see that they have the same summands with different coefficients $\pm 1$, and the coefficients for the later are all positive. Since each summand is a product of the $\lambda$-lengths, which is positive, the inequality follows. Since $\rho$ is non-Fuchsian, by (\ref{euler}), there must be an ideal triangle $t$ that has negative sign in $(\rho,d).$ Therefore, if $\gamma$ intersects $t,$ then some of the summands in the expression of $|tr\rho([\gamma])|$ has negative coefficients, and the inequality for $\gamma$ is strict.

For (2), choose arbitrarily  an ideal triangulation $\mathcal T$ of $\Sigma_{g,n},$ and let $T$ be the set of ideal triangles of $\mathcal T.$ By Theorems \ref{Kashaev1}, \ref{Kashaev2} and \ref{Kashaev3}, if $\mathcal M_k(\Sigma_{g,n})\neq\emptyset,$ then there exists $\epsilon\in\{\pm 1\}^T$ such that $\sum_{t\in T}\epsilon(t)=2k$ and the subset $\mathcal R(\mathcal T,\epsilon)$ is homeomorphic via the lengths coordinate to a full measure open subset $\Omega(\mathcal T,\epsilon)$ of $\mathbb R_{>0}^E.$ For each $\lambda\in\Omega(\mathcal T,\epsilon),$ let $(\rho,d)$ be the decorated representation determined by $(\lambda,\epsilon).$ Then $e(\rho)=k.$ On the other hand, $\mathbb R_{>0}^E$ is identified with the decorated Teichm\"uller space via the lengths coordinate, hence $\lambda$ determines a Fuchsian type-preserving representation $\rho'.$ By the same argument in (1), the inequality holds for $\rho$ and $\rho'
,$ and is strict for $\gamma$ intersecting the ideal triangles $t$ with $\epsilon(t)=-1.$
\end{proof}
 
 \begin{remark} It is very interesting to know if (2) holds for every Fuchsian type-preserving representation. This amounts to ask wether 
 $$\bigcup_{\mathcal T,\epsilon}\Omega(\mathcal T,\epsilon)=\mathbb R_{>0}^E,$$ 
 where the union is over all the ideal triangulations $\mathcal T$ of $\Sigma_{g,n}$ and all $\epsilon$ that gives the right relative Euler class.
 \end{remark}


\section{Tetrahedral triangulations}\label{Sec:4}

A \emph{tetrahedral triangulation} of the four-punctured sphere $\Sigma_{0,4}$ is an ideal triangulation of $\Sigma_{0,4}$ that is combinatorially equivalent to the boundary of an Euclidean tetrahedron (Figure \ref{Fig3}(a)). A pair of edges of a tetrahedral triangulation are called opposite if they are opposite edges of the tetrahedron. Let $v_1,\dots,v_4$ be the four punctures of $\Sigma_{0,4}.$ In the rest of this paper, for each tetrahedral triangulation $\mathcal T,$ we will let $t_i$ be the unique ideal triangle of $\mathcal T$ disjoint from the puncture $v_i$ and let $e_{ij}$ be the unique edge of $\mathcal T$ connecting the punctures $v_i$ and $v_j.$ We respectively denote by $x$ the pair of opposite edges $\{e_{12},e_{34}\},$ by $y$ the pair $\{e_{13},e_{24}\}$ and by $z$ the pair $\{e_{14},e_{23}\}.$ See Figure \ref{Fig3} (b).

 \begin{figure}[htbp]
\centering
\includegraphics[scale=0.25]{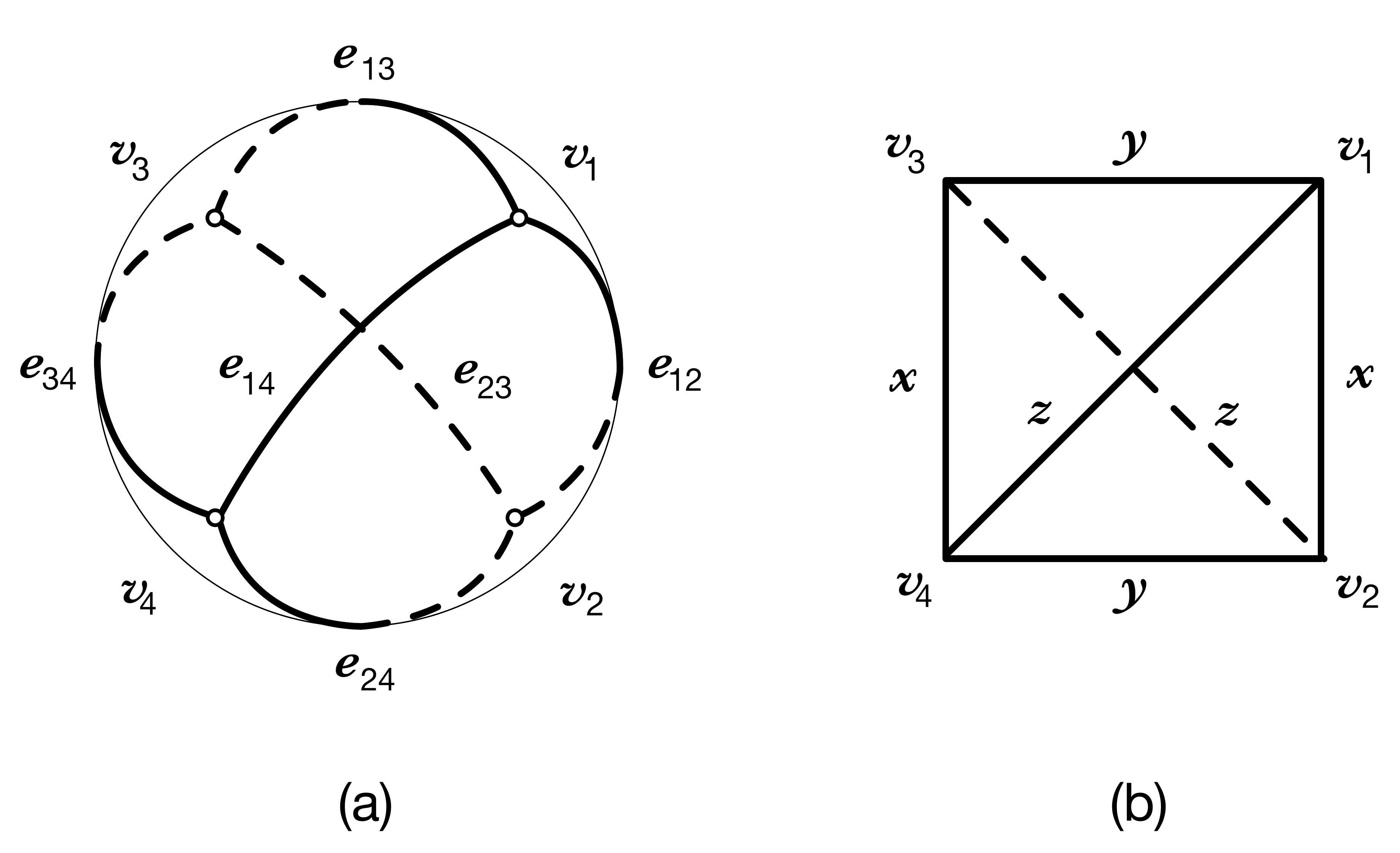}
\caption{}
\label{Fig3}
\end{figure}

A non-peripheral simple closed curve on $\Sigma_{0,4}$ is \emph{distinguished} in a tetrahedral triangulation $\mathcal T$ if it is disjoint from a pair of opposite edges of $\mathcal T$ and intersects each of the other four edges at exactly one point. In each tetrahedral triangulation, there are exactly three distinguished simple closed curves. We respectively denote by $X,$ $Y$ and $Z$ the distinguished simple closed curves disjoint from the pair of opposite edges $x,$ $y$ and $z.$ See Figure \ref{Fig4}. \begin{figure}[htbp]
\centering
\includegraphics[scale=0.30]{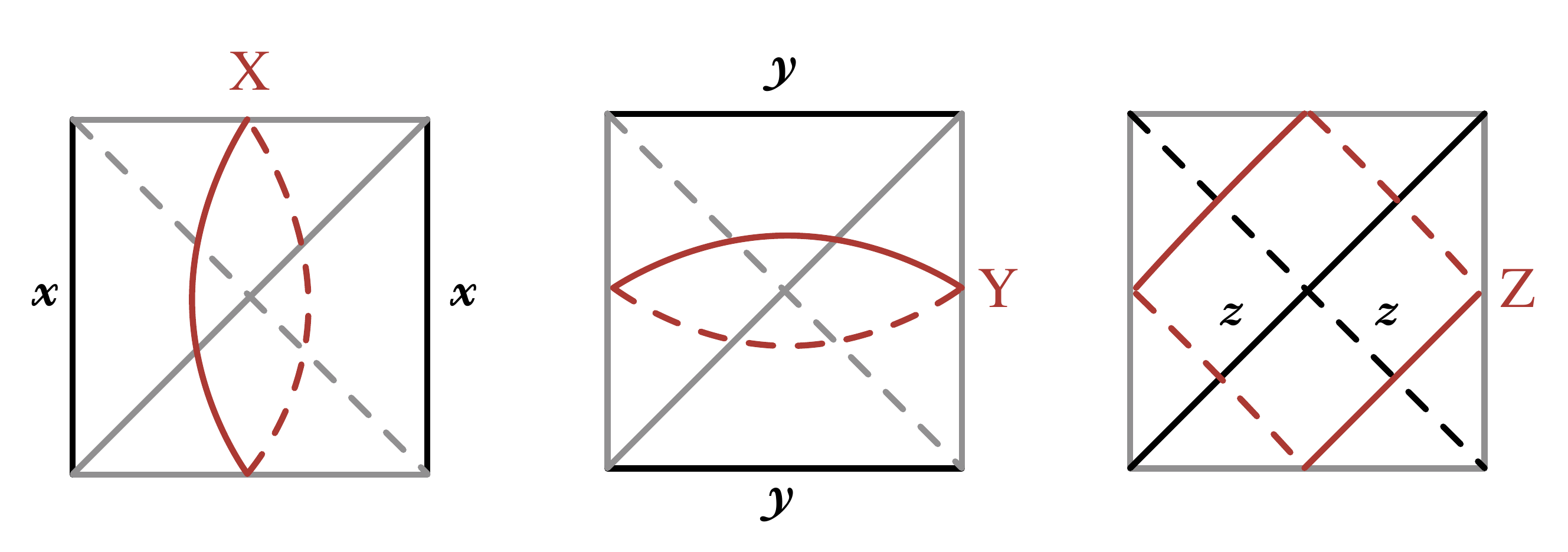}
\caption{}
\label{Fig4}
\end{figure}
The curves $X,$ $Y$ and $Z$ mutually intersect at exactly two points. On the other hand, for each triple of simple closed curves that mutually intersect at two points, there is a unique tetrahedral triangulation in which these three curves are distinguished. In particular, each non-peripheral simple closed curve on $\Sigma_{0,4}$ is distinguished in some tetrahedral triangulation. Note that being the $X$-, $Y$- or $Z$-curve is independent of the choice of the tetrahedral triangulation, since, for example, the curve $X$ always separates $\{v_1,v_2\}$ from $\{v_3,v_4\}.$ In the rest of this paper, we will call a simple closed curve an $X$- (resp. $Y$- or $Z$-) \emph{curve} if it is disjoint from the pair of opposite edges $x$ (resp. $y$ or $z$) of some tetrahedral triangulation. In this way, we get a tri-coloring of the set of non-peripheral simple closed curves on $\Sigma_{0,4}.$

A \emph{simultaneous diagonal switch} at a pair of opposite edges of $\mathcal T$ is an operation that  simultaneously does diagonal switches at this pair of edges. See Figure \ref{Fig6} (a). Denote respectively by $S_x,$ $S_y$ and $S_z$ the simultaneous diagonal switches at the pair of opposite edges $x,$ $y$ and $z.$ Then $S_x$ (reps. $S_y$ and $S_z$) changes the $X$- (resp. $Y$- and $Z$-) curve and leaves the other two distinguished simple closed curves unchanged. See Figure \ref{Fig6} (b).

 \begin{figure}[htbp]
\centering
\includegraphics[scale=0.30]{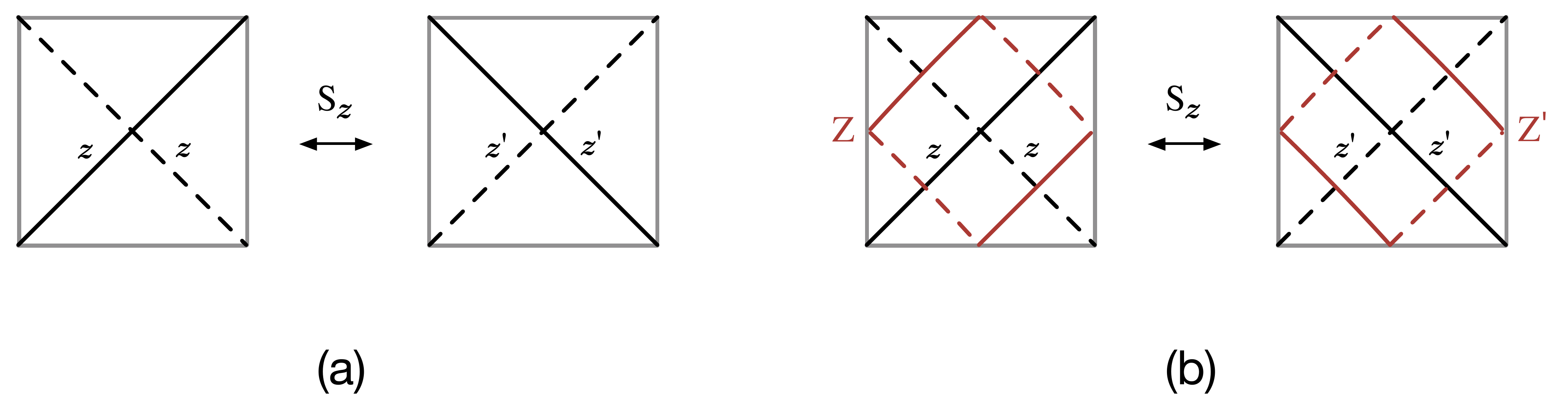}
\caption{}
\label{Fig6}
\end{figure}

The relationship between tetrahedral triangulations, simultaneous diagonal switches and non-peripheral simple closed curves can be described by (the dual of) the Farey diagram. Recall that the Farey diagram $\mathcal F$ is an ideal triangulation of $\mathbb H^2$ whose vertices are the extended rational numbers $\mathbb Q\cup\{\infty\}\subset\partial\mathbb H^2,$ and the dual Ferey diagram $\mathcal F^*$ is a countably infinite trivalent tree properly embedded in $\mathbb H^2.$ Each vertex of $\mathcal F$ corresponds to a non-peripheral simple closed curve on $\Sigma_{0,4},$ each edge of $\mathcal F$ connects two vertices corresponding to two simple closed curves that intersect at exactly two points and each ideal triangle of $\mathcal F$ corresponds to a triple of simple closed curves mutually intersecting at two points. (See e.g. \cite{KS}.) Therefore, each vertex of the dual graph $\mathcal F^*$ corresponds to a tetrahedral triangulation of $\Sigma_{0,4},$ each edge of $\mathcal F^*$ corresponds to a simultaneous diagonal switch and each connected component of $\mathbb H^2\setminus\mathcal F^*$ corresponds to a non-peripheral simple closed curves on $\Sigma_{0,4}.$ See Figure \ref{Fig7}. Since $\mathcal F^*$ is connected, any tetrahedral triangulation can be obtained from another by doing a finitely sequence of simultaneous diagonal switches. 

 \begin{figure}[htbp]
\centering
\includegraphics[scale=0.4]{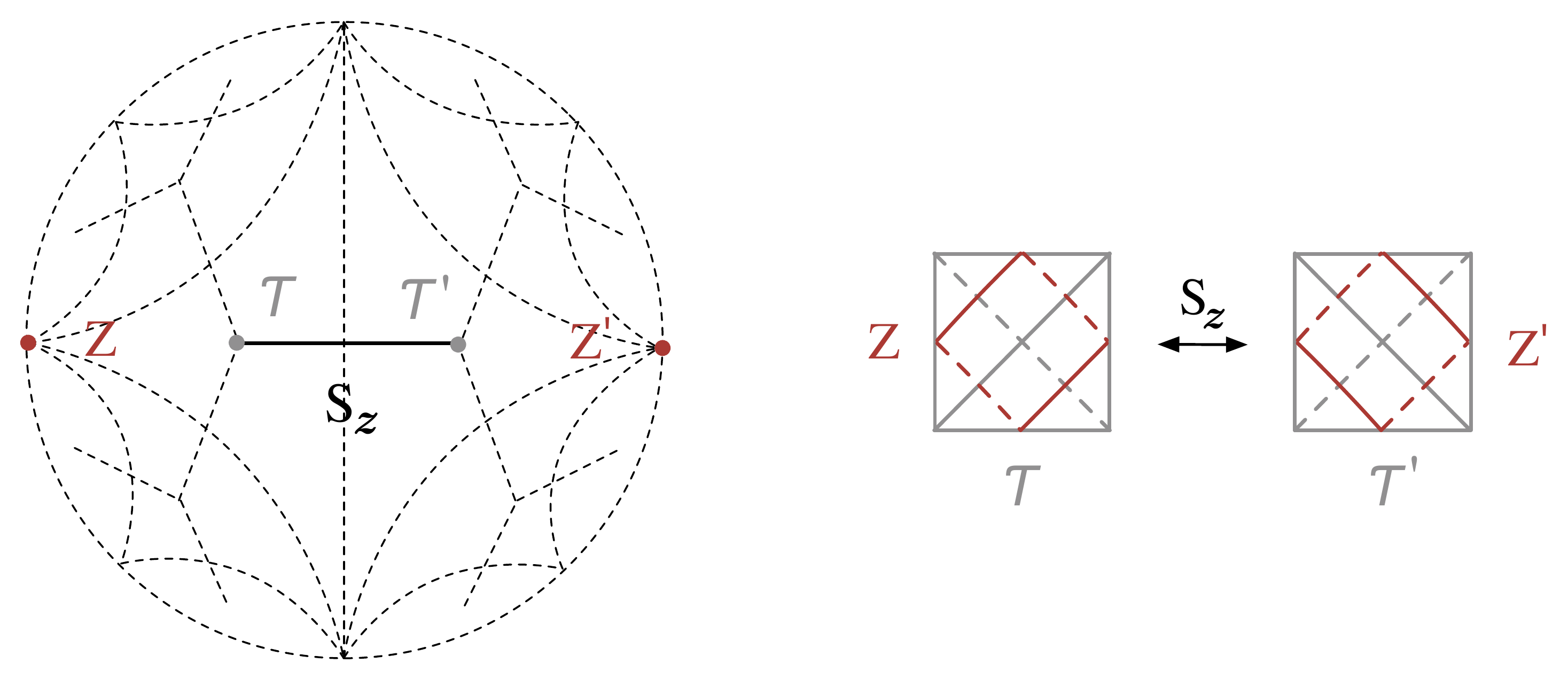}
\caption{}
\label{Fig7}
\end{figure}

We close up this section by showing the relationship between simultaneous diagonal switches and the mapping classes of $\Sigma_{0,4}.$

\begin{proposition}\label{evenodd}
A composition of an even number of simultaneous diagonal switches determines an element of $Mod(\Sigma_{0,4}).$ 
Conversely, any element of $Mod(\Sigma_{0,4})$ is determined by a composition of an even number of simultaneous diagonal switches. 
\end{proposition}

\begin{proof} Let $\mathcal T$ be a tetrahedra triangulation of $\Sigma_{0,4}.$ We write $\phi=S'S$ if $\phi$ is the self-diffeomorphism of $\Sigma_{0,4}$ such that the tetrahedral triangulation $\phi(\mathcal T)$ is obtained from $\mathcal T$ by doing the simultaneous diagonal switch $S$ followed by the simultaneous diagonal switch $S'.$ Then $D_X=S_zS_y$ and  $D_Y=S_xS_z.$ See Figure \ref{Fig8}.
 \begin{figure}[htbp]
\centering
\includegraphics[scale=0.30]{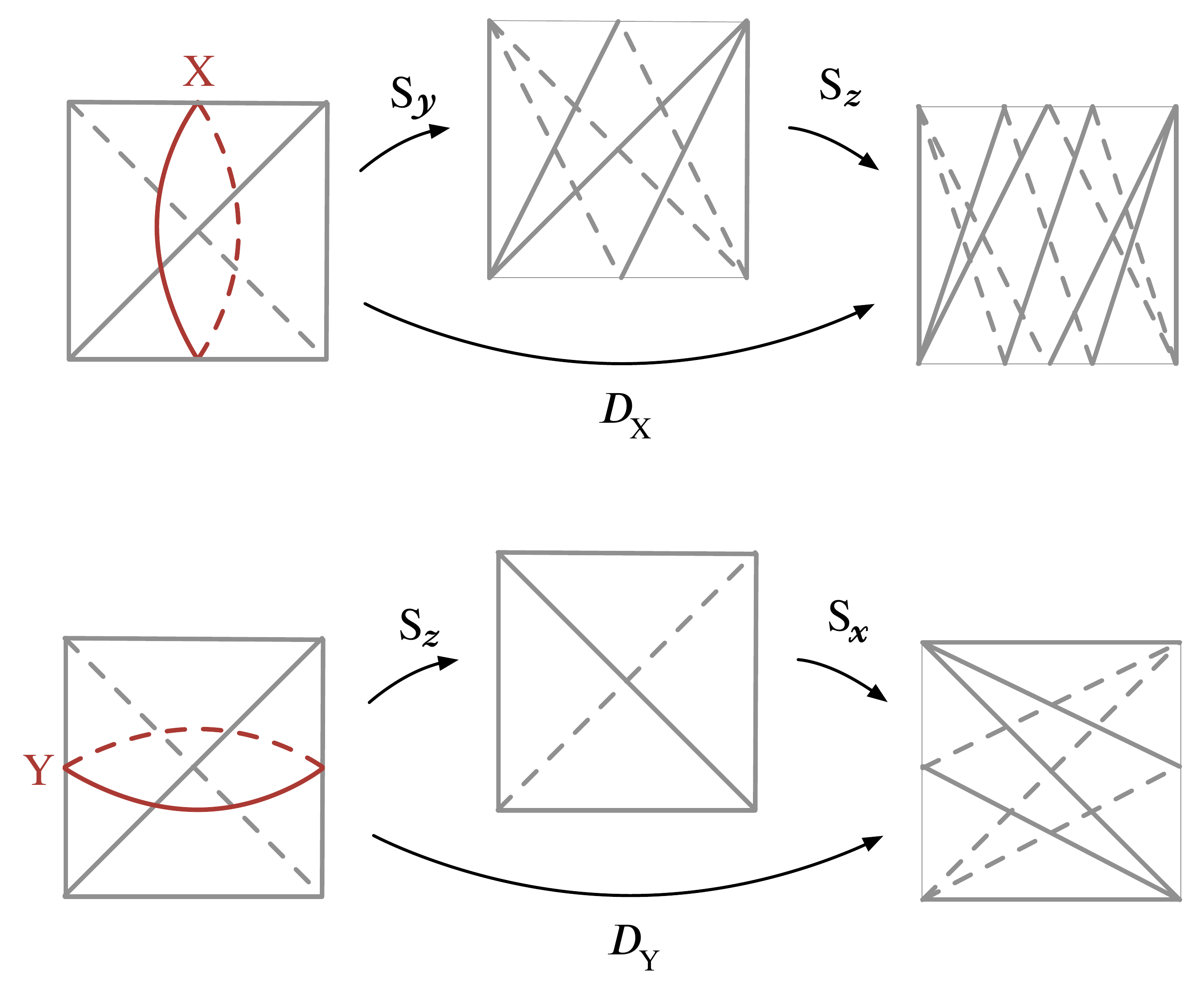}
\caption{}
\label{Fig8}
\end{figure}
Similarly, we have $S_yS_x=D_Z,$ $S_yS_z=D_X^{-1},$ $S_xS_z=D_Y^{-1}$ and $S_xS_y=D_X^{-1}.$ Thus, any composition of an even number of simultaneous diagonal switches determines an element of $Mod(\Sigma_{0,4}).$

For the converse statement, we define a cyclic order on the set $\{x,y,z\}$ of paris of opposite edges of $\mathcal T$ as follows. Since each puncture $v$ of $\Sigma_{0,4}$ is adjacent to three edges $e,$ $e'$ and $e''$ with $e\in x,$ $e'\in y$ and $e''\in z,$ the orientation of $\Sigma_{0,4}$ induces a cyclic order on the set $\{e,e',e''\}$ around $v,$ inducing a cyclic order on the set $\{x,y,z\}.$ It is easy to check that this cyclic order is independent of the choose of $v,$ hence is well defined. We call the \emph{sign} of a tetrahedral triangulation $\mathcal T$ positive if the cyclic order $x\mapsto y\mapsto z\mapsto x$ coincides with the one induced from the orientation, and negative if otherwise. It easy to see that a simultaneous diagonal switch changes the sign of $\mathcal T,$ and an orientation preserving self-diffeomorphism of $\Sigma_{0,4}$ preserves the sign of $\mathcal T.$ Since the dual Farey diagram $\mathcal F^*$ is a connected tree, for any self-diffeomorphism $\phi$ of $\Sigma_{0,4},$ up to redundancy there is a unique path of $\mathcal F^*$ connecting the vertices $\mathcal T$ and $\phi(\mathcal T).$ Since $\mathcal T$ and $\phi(\mathcal T)$ have the same sign, the path consists of an even number of edges, corresponding to an even number of simultaneous diagonal switche $S_1,\dots, S_{2m}.$ Then $\phi=\phi_k\circ\dots\circ\phi_1,$ where $\phi_k=S_{2k}S_{2k-1}.$ 
\end{proof}

\section{Bowditch's question}\label{Sec:5}

Let $\rho$ be a type-preserving representation of $\pi_1(\Sigma_{0,4})$ and let $d$ be a decoration of $\rho.$ Suppose $\mathcal T$ is a $\rho$-admissible tetrahedral triangulation of $\Sigma_{0,4},$ $E$ and $T$ respectively are the sets of edges and ideal triangles of $\mathcal T,$ and $(\lambda,\epsilon)\in\mathbb R_{>0}^E\times\{\pm1\}^T$ is the lengths coordinate of $[(\rho,d)]\in\mathcal M^d_{\pm1}(\Sigma_{0,4}).$  Let $v_1,\dots,v_4$ be the punters of $\Sigma_{0,4},$ let $t_i$ be the ideal triangle of $\mathcal T$ disjoint from $v_i$ and let $e_{ij}$ be the edge of $\mathcal T$ connecting the punctures $v_i$ and $v_j.$  Define the quantities $\lambda(x)=\lambda(e_{12})\lambda(e_{34}),$ $\lambda(y)=\lambda(e_{13})\lambda(e_{24})$ and $\lambda(z)=\lambda(e_{14})\lambda(e_{23}).$ The quantities $\lambda(x),$ $\lambda(y)$ and $\lambda(z)$ will play a central role in the rest of this paper. 


\subsection{A proof of Theorem \ref{main1}}

Suppose $e(\rho)=1.$ Then by (\ref{euler}), there is exactly one ideal triangle, say $t_1,$ such that $\epsilon(t_1)=-1$ and $\epsilon(t_i)=1$ for $i\neq 1.$ As a direct consequence of Lemma \ref{3.1} and Theorem \ref{trace formula}, we have the following lemmas. 

\begin{lemma}\label{5.1} Let $\gamma_i$ be the simple closed curve going counterclockwise around the puncture $v_i$ once. Then up to conjugation, the $\rho$-image of the peripheral element $[\gamma_1]\in\pi_1(\Sigma_{0,4})$ is 
\begin{equation*}
\pm\left[
      \begin{array}{cc}
      1 & \lambda(x)+\lambda(y)+\lambda(z)\\
    0&1
      \end{array} \right],
 \end{equation*} and the $\rho$-image of the other peripheral elements $[\gamma_2],$ $[\gamma_3]$ and $[\gamma_4]$ are respectively \begin{equation*}
      \pm\left[
      \begin{array}{cc}
      1 & \lambda(y)+\lambda(z)-\lambda(x)\\
    0&1
      \end{array} \right],
      \pm\left[
      \begin{array}{cc}
      1 & \lambda(x)+\lambda(z)-\lambda(y)\\
    0&1
      \end{array} \right] \text{ and }
      \pm\left[
      \begin{array}{cc}
      1 & \lambda(x)+\lambda(y)-\lambda(z)\\
    0&1
      \end{array} \right].
  \end{equation*}
\end{lemma}

\begin{lemma}\label{5.2}
\begin{enumerate}[(1)]
\item The absolute values of the traces of the distinguished simple closed curves $X,$ $Y$ and $Z$ of $\mathcal T$ can be calculated by
\begin{equation}\label{5.1}
\begin{split}
\big|tr\rho([X])\big|&=\frac{\big|\lambda(y)^2+\lambda(z)^2-\lambda(x)^2\big|}{\lambda(y)\lambda(z)},\\
\big|tr\rho([Y])\big|&=\frac{\big|\lambda(x)^2+\lambda(z)^2-\lambda(y)^2\big|}{\lambda(x)\lambda(z)}\ \text{ and}\\
\big|tr\rho([Z])\big|&=\frac{\big|\lambda(x)^2+\lambda(y)^2-\lambda(z)^2\big|}{\lambda(x)\lambda(y)}.
\end{split}
\end{equation}

\item The right hand sides of the equations in (\ref{5.1}) are strictly greater than $2$ if and only if $\lambda(x),$ $\lambda(y)$ and $\lambda(z)$ satisfy one of the following inequalities
\begin{equation}\label{anti-tri}
\begin{split}
\lambda(x)&>\lambda(y)+\lambda(z),\\
\lambda(y)&>\lambda(x)+\lambda(z)\ \text{ or}\\
\lambda(z)&>\lambda(x)+\lambda(y).
\end{split}
\end{equation}
\end{enumerate}
\end{lemma}

Note that reversing the directions of the inequalities in (\ref{anti-tri}), we get the triangular inequality. The idea of the proof of (2) is that if we regard the quantities $\lambda(x),$ $\lambda(y)$ and $\lambda(z)$ as the edge lengths of a Euclidean triangle, then the right hand sides of (\ref{5.1}) are twice of the cosine of the corresponding inner angles. 
The next lemma shows the rule of the change of the quantities $\lambda(x),$ $\lambda(y)$ and $\lambda(z)$ under a simultaneous diagonal switch.

\begin{lemma}\label{lambda'} Suppose $\mathcal T'$ is a tetrahedral triangulation of $\Sigma_{0,4}.$  If $\mathcal T'$ is $\rho$-admissible,
then let $\lambda'$ be the $\lambda$-lengths of $(\rho,d)$ in $\mathcal T',$ and let $\lambda'(x),$ $\lambda'(y)$ and $\lambda'(z)$ be the corresponding quantities.
\begin{enumerate}[(1)]
\item If $\mathcal T'$ is obtained from $\mathcal T$ by doing $S_x,$ then $\mathcal T'$ is $\rho$-admissible if and only if $\lambda(y)\neq\lambda(z).$ In the case that $\mathcal T'$ is $\rho$-admissible,  $\lambda'(y)=\lambda(y),$ $\lambda'(z)=\lambda(z)$ and 
\begin{equation*}
\lambda'(x)=\frac{\big|\lambda(y)^2-\lambda(z)^2\big|}{\lambda(x)}.
\end{equation*}
\item If $\mathcal T'$ is obtained from $\mathcal T$ by doing $S_y,$ then $\mathcal T'$ is $\rho$-admissible if and only if $\lambda(x)\neq\lambda(z).$ In the case that  $\mathcal T'$ is $\rho$-admissible, $\lambda'(x)=\lambda(x),$ $\lambda'(z)=\lambda(z)$ and 
\begin{equation*}
\lambda'(y)=\frac{\big|\lambda(z)^2-\lambda(x)^2\big|}{\lambda(y)}.
\end{equation*}
\item If $\mathcal T'$ is obtained from $\mathcal T$ by doing $S_z,$ then $\mathcal T'$ is $\rho$-admissible if and only if $\lambda(x)\neq\lambda(y).$ In the case that  $\mathcal T'$ is $\rho$-admissible, $\lambda'(x)=\lambda(x),$ $\lambda'(y)=\lambda(y)$ and 
\begin{equation*}
\lambda'(z)=\frac{\big|\lambda(x)^2-\lambda(y)^2\big|}{\lambda(z)}.
\end{equation*}
\end{enumerate}
\end{lemma}

\begin{proof} For (1), we have that the edge $e_{12}$  is adjacent to the ideal triangle $t_3$ and $t_4$ with $\epsilon(t_3)=\epsilon(t_4)$ and  $e_{34}$ is adjacent to the ideal triangles $t_1$ and $t_2$ with $\epsilon(t_1)\neq\epsilon(t_2).$ Let $e_{34}'$ and $e_{12}'$ respectively be the edges of $\mathcal T'$ obtained from diagonal switches at $e_{12}$ and $e_{34},$ i.e., $e'_{12}$ is the edge of $\mathcal T'$ connecting the punctures $v_1$ and $v_2$ and $e'_{34}$ is the edge of $\mathcal T'$ connecting the punctures $v_3$ and $v_4.$ By Proposition \ref{change}, $\mathcal T'$ is $\rho$-admissible if and only if $\lambda(e_{13})\lambda(e_{24})\neq\lambda(e_{14})\lambda(e_{23}),$ i.e., $\lambda(y)\neq\lambda(z).$ By Proposition \ref{change} again, if $\mathcal T'$ is $\rho$-admissible, then $\lambda'(e_{ij})=\lambda(e_{ij})$ for $\{i,j\}\neq\{1,2\}$ or $\{3,4\},$ and
$$\lambda'(e_{12}')=\frac{\big|\lambda(e_{13})\lambda(e_{24})-\lambda(e_{14})\lambda(e_{23})\big|}{\lambda(e_{12})}$$and
$$\lambda'(e_{34}')=\frac{\lambda(e_{13})\lambda(e_{24})+\lambda(e_{14})\lambda(e_{23})}{\lambda(e_{34})}.$$
Therefore, $\lambda'(y)=\lambda(y),$ $\lambda'(z)=\lambda(z)$ and $\lambda'(x)=\lambda'(e_{12}')\lambda'(e_{34}')={\big|\lambda(y)^2-\lambda(z)^2\big|}/{\lambda(x)}.$ 

The proofs of (2) and (3) are the similar. \end{proof}

A consequence of Lemma \ref{lambda'} is that the inequalities in (\ref{anti-tri}) are persevered by the simultaneous diagonal switches. 

\begin{lemma}\label{preserve} Suppose $\mathcal T'$ is a $\rho$-admissible tetrahedral triangulation of $\Sigma_{0,4}$ obtained from $\mathcal T$ by doing a simultaneous diagonal switch. Let $\lambda'$ be the $\lambda$-lengths of $(\rho,d)$ in $\mathcal T',$ and let $\lambda'(x),$ $\lambda'(y)$ and $\lambda'(z)$ be the corresponding quantities. Then $\lambda'(x)$, $\lambda'(y)$ and $\lambda'(z)$ satisfy one of the inequalities in (\ref{anti-tri}) if and only if $\lambda(x),$ $\lambda(y)$ and $\lambda(z)$ do.
\end{lemma}

\begin{proof} Without lost of generality, we assume that $\mathcal T'
$ is obtained from $\mathcal T$ by doing $S_x.$
If $\lambda(x)>\lambda(y)+\lambda(z),$ then by Lemma \ref{lambda'},
$$\lambda'(x)=\frac{\big|\lambda(y)^2-\lambda(z)^2\big|}{\lambda(x)}<\frac{\big|\lambda(y)^2-\lambda(z)^2\big|}{\lambda(y)+\lambda(z)}=\big|\lambda(y)-\lambda(z)\big|=\big|\lambda'(y)-\lambda'(z)\big|.$$
Therefore, either $\lambda'(y)>\lambda'(x)+\lambda'(z)$ or $\lambda'(z)>\lambda'(x)+\lambda'(y).$ On the other hand, if either $\lambda(y)>\lambda(x)+\lambda(z)$ or $\lambda(z)>\lambda(x)+\lambda(y),$ i.e., $\lambda(x)<\big|\lambda(y)-\lambda(z)\big|,$ then by Lemma \ref{lambda'},
$$\lambda'(x)=\frac{\big|\lambda(y)^2-\lambda(z)^2\big|}{\lambda(x)}>\frac{\big|\lambda(y)^2-\lambda(z)^2\big|}{\big|\lambda(y)-\lambda(z)\big|}=\lambda(y)+\lambda(z)=\lambda'(y)+\lambda'(z).$$
\end{proof}

Another consequence of Lemma \ref{lambda'} is the following

\begin{proposition}\label{5.5} There are uncountably many $[\rho]\in\mathcal M_{\pm1}(\Sigma_{0,4})$ such that all the tetrahedral triangulations of $\Sigma_{0,4}$ are $\rho$-admissible. 
\end{proposition}

\begin{proof} Suppose $\rho$ is a typer-preserving representation of $\pi_1(\Sigma_{0,4})$ with $e(\rho)=1,$ and $d$ is a decoration of $\rho.$ Let $\mathcal T$ be a $\rho$-admissible tetrahedral triangulation of $\Sigma_{0,4}$ and let $(\lambda,\epsilon)$ be the lengths coordinate of $(\rho,d)$ in $\mathcal T.$ Recall that there is a ono-to-one correspondence between the tetrahedral triangulations of $\Sigma_{0,4}$ and the vertices of the dual Farey diagram $\mathcal F^*,$ which is a countably infinity tree. Therefore, for each tetrahedral triangulation $\mathcal T',$ there is up to redundancy a unique path in $\mathcal F^*$ connecting $\mathcal T$ and $\mathcal T',$ which corresponds to a sequence $\{S_{i}\}_{i=1}^n$ of simultaneous diagonal switches. Let $\mathcal T_0=\mathcal T,$ and for each $i\in\{1,\dots,n\},$ let $\mathcal T_{i}$ be the tetrahedral triangulation obtained from $\mathcal T_{i-1}$ by doing $S_{i}.$ Suppose $\mathcal T_{i}$ is $\rho$-admissible for some $i\in\{1,\dots,n\},$ and suppose $\lambda_{i}$ is the $\lambda$-lengths of $(\rho,d)$ in $\mathcal T_{i}.$ Then by Lemma \ref{lambda'}, $\mathcal T_{i+1}$ is $\rho$-admissible if and only if the Laurent polynomial $\frac{\lambda_{i}(y)^2-\lambda_{i}(z)^2}{\lambda_{i}(x)}\neq0.$ An induction in $i$ shows that $\mathcal T'=\mathcal T_n$ is $\rho$-admissible if and only if certain Laurent polynomial $L_{\mathcal T'}(\lambda(x),\lambda(y),\lambda(z))\neq0.$ The set of zeros $Z_{\mathcal T'}$ of $L_{\mathcal T'}$ is a Zariski-closed subset of $\mathbb R_{>0}^3.$ In particular, the Lebesgue measure $m(Z_{\mathcal T'})=0.$ Since $\mathcal F^*$ is a countably infinity tree, there are in total countably many tetrahedral triangulations $\mathcal T$ of $\Sigma_{0,4},$ and hence $m(\bigcup_{\mathcal T}Z_{\mathcal T})=0.$ Therefore, the set $\mathcal C=\mathbb R^3_{>0}\setminus\bigcup_{\mathcal T}Z_{\mathcal T}$ has a full measure in $\mathbb R_{>0}^3.$ In particular, $\mathcal C$ contains uncountable many points. 

Now each $(a,b,c)\in\mathcal C$ with $a+b\neq c,$ $a+c\neq b$ and $b+c\neq a$ determines a type-preserving representation $\rho$ as follows. Take a tetrahedral triangulation $\mathcal T$ of $\Sigma_{0,4},$ and let $E$ and $T$ respectively be the set of edges and ideal triangles of $\mathcal T.$ Choose $\epsilon\in\{\pm 1\}^T$ so that $\sum_{t\in T}\epsilon(t)=2,$ and define $\lambda\in\mathbb R_{>0}^E$ by $\lambda(e_{12})=\lambda(e_{34})=a^{\frac{1}{2}},$ $\lambda(e_{13})=\lambda(e_{24})=b^{\frac{1}{2}}$ and $\lambda(e_{14})=\lambda(e_{23})=c^{\frac{1}{2}}.$ Then $\lambda(x)=a,$ $\lambda(y)=b$ and $\lambda(z))=c.$ By Theorem \ref{Kashaev2}, Proposition \ref{rational} and Lemma \ref{5.1}, $(\lambda,\epsilon)$ determines a decorated representation $(\rho,d)$ up to conjugation. In particular, by Lemma \ref{5.1}, $\rho$ is type-preserving. By (\ref{euler}), the relative Euler class $e(\rho)=1.$ Finally, since $(\lambda(x),\lambda(y),\lambda(z))\in\mathcal C,$ the Laurent polynomial $L_{\mathcal T'}(\lambda(x),\lambda(y),\lambda(z))\neq 0$ for all tetrahedral triangulation $\mathcal T'.$ As a consequence, all the tetrahedral triangulations are $\rho$-admissible. 

By symmetry, there are also uncountably many type-preserving representations $\rho$ with $e(\rho)=-1$ such that all the tetrahedral triangulations of $\Sigma_{0,4}$ are $\rho$-admissible. 
\end{proof}

\begin{proof}[Proof of Theorem \ref{main1}] Let $\mathcal T$ be a tetrahedral triangulation of $\Sigma_{0,4}$ and let $\mathcal C$ be the full measure subset of $\mathbb R_{>0}^3$ constructed in the proof of Proposition \ref{5.5}. Then each $(a,b,c)\in\mathcal C$ satisfying one of the following identities $a>b+c,$ $b>a+c$ or $c>a+b$ determines a decorated representation $(\rho,d)$ with $e(\rho)=\pm1$ such that all the tetrahedral triangulations of $\Sigma_{0,4}$ are $\rho$-admissible. Since elementary representation have relative Euler class $0$ and $e(\rho)=\pm1,$ $\rho$ is non-elementary. For each tetrahedral triangulation $\mathcal T',$  let $\lambda'$ be the $\lambda$-lengths of $(\rho,d)$ in $\mathcal T'.$ Since $\mathcal T'
$ can by obtained from $\mathcal T$ by doing a sequence of simultaneous diagonal switches, by Lemma \ref{preserve}, the quantities $\lambda'(x),$ $\lambda'(y)$ and $\lambda'(z)$ satisfy one of the inequalities in (\ref{anti-tri}). By Lemma \ref{5.2}, the traces of the distinguished simple closed curves $X,$ $Y$ and $Z$ in $\mathcal T'
$ are strictly greater than $2$ in the absolute value. Since each simple closed curve $\gamma$ is distinguished in some tetrahedral triangulation $\mathcal T',$ we have $\big|tr\rho([\gamma])\big|>2.$
\end{proof}


\subsection{A proof of Theorem \ref{main2}}

Suppose $e(\rho)=0.$ Then by (\ref{euler}), there are exactly two ideal triangles having the positive sign and two having the negative sign. Without loss of generality, we assume that  $\epsilon(t_1)=\epsilon(t_2)=-1$ and $\epsilon(t_3)=\epsilon(t_4)=1.$ Note that under this assumption, the edges $e_{12}$ and $e_{34}$ in the pair $x$ are adjacent to ideal triangles having the same sign, and as will be seen later, the $X$-curves will play a different role than the $Y$- and $Z$- curves do. As a direct consequence of Lemma \ref{3.1}, we have the following 

\begin{lemma}\label{57} Let $\gamma_i$ be the simple closed going counterclockwise around the puncture $v_i.$ Then up to conjugation, the $\rho$-image of the peripheral elements $[\gamma_1]$ and $[\gamma_2]$ of $\pi_1(\Sigma_{0,4})$ are
\begin{equation*}
\pm\left[
      \begin{array}{cc}
      1 & \lambda(y)+\lambda(z)-\lambda(x)\\
    0&1
      \end{array} \right],
 \end{equation*}
 and the $\rho$-image of the other two peripheral elements $[\gamma_3]$ and $[\gamma_4]$ are  \begin{equation*}
      \pm\left[
      \begin{array}{cc}
      1 & \lambda(x)-\lambda(y)-\lambda(z)\\
    0&1
      \end{array} \right].
  \end{equation*}
\end{lemma}

\begin{lemma}\label{5.8}
\begin{enumerate}[(1)]
\item The absolute values of the traces of the distinguished simple closed curves $X,$ $Y$ and $Z$ of $\mathcal T$ can be calculated by
\begin{equation}\label{5.3}
\begin{split}
\big|tr\rho([X])\big|&=\frac{\big|\lambda(x)^2+\lambda(y)^2+\lambda(z)^2-2\lambda(x)\lambda(y)-2\lambda(x)\lambda(z)\big|}{\lambda(y)\lambda(z)},\\
\big|tr\rho([Y])\big|&=\frac{\lambda(x)^2+\lambda(y)^2+\lambda(z)^2+2\lambda(y)\lambda(z)-2\lambda(x)\lambda(y)}{\lambda(x)\lambda(z)}\ \text{ and}\\
\big|tr\rho([Z])\big|&=\frac{\lambda(x)^2+\lambda(y)^2+\lambda(z)^2+2\lambda(y)\lambda(z)-2\lambda(x)\lambda(z)}{\lambda(x)\lambda(y)}.
\end{split}
\end{equation}

\item The right hand sides of the last two equations in (\ref{5.3}) are always strictly greater than $2,$ whereas the right hand side of the first equation is less than or equal to $2$ if and only if $\lambda(x),$ $\lambda(y)$ and $\lambda(z)$ satisfy the following inequalities
\begin{equation}\label{tri}
\begin{cases}
\sqrt{\lambda(x)}\leqslant\sqrt{\lambda(y)}+\sqrt{\lambda(z)},\\
\sqrt{\lambda(y)}\leqslant\sqrt{\lambda(x)}+\sqrt{\lambda(z)}\ \text{ and}\\
\sqrt{\lambda(z)}\leqslant\sqrt{\lambda(x)}+\sqrt{\lambda(y)}.
\end{cases}
\end{equation}
\end{enumerate}
\end{lemma}

\begin{proof} (1) is a direct consequence of Theorem \ref{trace formula}. For (2), since $\rho$ is type-preserving, by Theorem \ref{Kashaev2}, Proposition \ref{rational} and Lemma \ref{57}, $\lambda(x)-\lambda(y)-\lambda(z)\neq0.$ Therefore, the right hand side of the second equation of (\ref{5.3}) equals $\frac{\big(\lambda(x)-\lambda(y)-\lambda(z)\big)^2}{\lambda(x)\lambda(z)}+2>2,$ and the right hand side of the third equation equals $\frac{\big(\lambda(x)-\lambda(y)-\lambda(z)\big)^2}{\lambda(x)\lambda(y)}+2>2.$ In the case that $\lambda(x)^2+\lambda(y)^2+\lambda(z)^2-2\lambda(x)\lambda(y)-2\lambda(x)\lambda(z)\geqslant 0,$ the right hand side of the first equation of (\ref{5.3}) equals $\frac{\big(\lambda(x)-\lambda(y)-\lambda(z)\big)^2}{\lambda(y)\lambda(z)}-2>-2.$ The quantity also equals $$2-\frac{(\lambda(x)^{\frac{1}{2}}+{\lambda(y)}^{\frac{1}{2}}+{\lambda(z)}^{\frac{1}{2}})({\lambda(x)}^{\frac{1}{2}}+{\lambda(y)}^{\frac{1}{2}}-{\lambda(z)}^{\frac{1}{2}})({\lambda(x)}^{\frac{1}{2}}+{\lambda(z)}^{\frac{1}{2}}-{\lambda(y)}^{\frac{1}{2}})({\lambda(y)}^{\frac{1}{2}}+{\lambda(z)}^{\frac{1}{2}}-{\lambda(x)}^{\frac{1}{2}})}{\lambda(y)\lambda(z)},$$
which is less than or equal to $2$ if and only if the equalities in (\ref{tri}) are satisfied. For the case that $\lambda(x)^2+\lambda(y)^2+\lambda(z)^2-2\lambda(x)\lambda(y)-2\lambda(x)\lambda(z)\leqslant 0,$ the proof is similar.
\end{proof}

The next lemma shows the rule of the change of the quantities $\lambda(x),$ $\lambda(y)$ and $\lambda(z)$ under a simultaneous diagonal switch.

\begin{lemma}\label{lambda2} Suppose $\mathcal T'$ is a tetrahedral triangulation of $\Sigma_{0,4}.$ If $\mathcal T'$ is $\rho$-admissible, then let $\lambda'$ be the $\lambda$-lengths of$(\rho,d)$ in $\mathcal T',$ and let $\lambda'(x),$ $\lambda'(y)$ and $\lambda'(z)$ be the corresponding quantities.

\begin{enumerate}[(1)]
\item If $\mathcal T'$ is obtained from $\mathcal T$ by doing $S_x,$ then $\mathcal T'$ is $\rho$-admissible. In this case, $\lambda'(y)=\lambda(y),$ $\lambda'(z)=\lambda(z)$ and 
\begin{equation*}
\lambda'(x)=\frac{\big(\lambda(y)+\lambda(z)\big)^2}{\lambda(x)}.
\end{equation*}
\item If $\mathcal T'$ is obtained from $\mathcal T$ by doing $S_y,$ then $\mathcal T'$ is $\rho$-admissible if and only if $\lambda(x)\neq\lambda(z).$ In the case that $\mathcal T'$ is $\rho$-admissible, $\lambda'(x)=\lambda(x),$ $\lambda'(z)=\lambda(z)$ and 
\begin{equation*}
\lambda'(y)=\frac{\big(\lambda(z)-\lambda(x)\big)^2}{\lambda(y)}.
\end{equation*}
\item If $\mathcal T'$ is obtained from $\mathcal T$ by doing $S_z,$ then $\mathcal T'$ is $\rho$-admissible if and only if $\lambda(x)\neq\lambda(z).$ In the case that $\mathcal T$ is $\rho$-admissible, $\lambda'(x)=\lambda(x),$ $\lambda'(y)=\lambda(y)$ and 
\begin{equation*}
\lambda'(z)=\frac{\big(\lambda(x)-\lambda(y)\big)^2}{\lambda(z)}.
\end{equation*}
\end{enumerate}
\end{lemma}
\begin{proof} This is a consequence of Proposition \ref{change}, and the proof is similar to that of Lemma \ref{lambda'}. \end{proof}

\begin{proof}[Proof of Theorem \ref{main2}] Let $\rho$ be a type-preserving representation of $\pi_1(\Sigma_{0,4})$ with relative Euler class $e(\rho)=0,$ and choose arbitrarily a decoration $d$ of $\rho.$ Let $\mathcal T$ be a tetrahedral triangulation of $\Sigma_{0,4}.$ If $\mathcal T$ is not $\rho$-admissible, then there is an edge $e$ of $\mathcal T$ that is not $\rho$-admissible, and the element of $\pi_1(\Sigma_{0,4})$ represented by the distinguished simple closed curve in $\mathcal T$ disjoint from $e$ is sent by $\rho$ to a parabolic element of $PSL(2,\mathbb R).$ If $\mathcal T$ is $\rho$-admissible, then we let $(\lambda,\epsilon)$ be the lengths coordinate of $(\rho, d)$ in $\mathcal T.$ If the quantities $\lambda(x),$ $\lambda(y)$ and $\lambda(z)$ satisfy the inequalities in (\ref{tri}), then by Lemma \ref{5.8}, the element of $\pi_1(\Sigma_{0,4})$ represented by one of the distinguished simple closed curves $X,$ $Y$ and $Z$ is sent by $\rho$ to either an elliptic or a parabolic element of $PSL(2,\mathbb R).$ Therefore, to prove the theorem, it suffices to find a tetrahedral triangulation $\mathcal T'$ of $\Sigma_{0,4}$ such that either $\mathcal T'$ is not $\rho$-admissible or $\mathcal T'$ is $\rho$-admissible with the quantities $\lambda'(x),$ $\lambda'(y)$ and $\lambda'(z)$ satisfying the inequalities in (\ref{tri}). Our strategy of finding $\mathcal T'$ is to construct a sequence of tetrahedral triangulations $\{\mathcal T_n\}_{n=1}^N$ with $\mathcal T_N=\mathcal T'$ by the following 
\medskip

{\it Trace Reduction Algorithm:} Let $\mathcal T_0=\mathcal T$ and suppose that $\mathcal T_n$ is obtained. If $\mathcal T_n$ is not $\rho$-admissible, then we stop. If $\mathcal T_n$ is $\rho$-admissible, then we let $(\lambda_n,\epsilon_n)$ be the lengths coordinate of $(\rho,d)$ in $\mathcal T_n.$ If $\lambda_n(x),$ $\lambda_n(y)$ and $\lambda_n(z)$ satisfy the inequalities in (\ref{tri}), then we stop. If otherwise, then there is a unique maximum among $\lambda_n(x),$ $\lambda_n(y)$ and $\lambda_n(z),$ since other wise the inequalities (\ref{tri}) are satisfied. Suppose $\{e_{ij},e_{kl}\}$ is the pair of opposite edges of $\mathcal T_n$ such that $\lambda(e_{ij})\lambda(e_{kl})$ equals the maximum of $\lambda_n(x),$ $\lambda_n(y)$ and $\lambda_n(z).$ Then we let $\mathcal T_{n+1}$ be the tetrahedral triangulation obtained from $\mathcal T_n$ by doing a simultaneous diagonal switch at $e_{ij}$ and $e_{kl}.$ 
\medskip

By Lemma \ref{finite} below, the algorithm stops  at some $\mathcal T_N.$ 
\end{proof}

\begin{lemma}\label{finite} The Trace Reduction Algorithm stops in finitely many steps. 
\end{lemma}

\begin{proof} For each $n,$ let $t_i$ be the ideal triangle of $\mathcal T_n$ disjoint from the puncture $v_i,$ and let $e_{ij}$ be the edge of $\mathcal T_n$ connecting the punctures $v_i$ and $v_j.$  Without loss of generality, we assume in $\mathcal T$ that $\epsilon(t_1)=\epsilon(t_2)=-1$ and $\epsilon(t_3)=\epsilon(t_4)=1.$ Then by Proposition \ref{change}, $\epsilon_n(t_1)=\epsilon_n(t_2)$ and $\epsilon_n(t_3)=\epsilon_n(t_4)$ for each $\mathcal T_n.$ For each $n,$ we let $a_n=\frac{\lambda_n(x)}{\lambda_n(x)+\lambda_n(y)+\lambda(z)},$ $b_n=\frac{\lambda_n(y)}{\lambda_n(x)+\lambda_n(y)+\lambda_n(z)}$ and $c_n=\frac{\lambda_n(z)}{\lambda_n(x)+\lambda_n(y)+\lambda_n(z)},$ and let $$k_n=\max\big\{\sqrt{a_n}-\sqrt{b_n}-\sqrt{c_n},\
\sqrt{b_n}-\sqrt{a_n}-\sqrt{c_n},\ 
\sqrt{c_n}-\sqrt{a_n}-\sqrt{b_n}\big\}.$$ Then $\lambda_n(x),$ $\lambda_n(y)$ and $\lambda_n(z)$ satisfy the inequalities in (\ref{tri}) if and only if $k_n\leqslant 0.$

Assume that the sequence $\{\mathcal T_n\}$ is infinite, i.e., $k_n>0$ for all $n>0.$ Then we will find a contradiction by the following three steps. In Step I we show that $k_n$ is decreasing in $n$ by considering two mutually complementary cases, where in one of them (Case 1) the gap $k_n-k_{n+1}$ is bounded below by the minimum of $a_n, b_n$ and $c_n.$ In Step II we show that there must be a infinite subsequence $\{\mathcal T_{n_i}\}$ of $\{\mathcal T_n\}$ such that each $\mathcal T_{n_i}$ is of Case 1 of Step I, and in Step III we show that for $i$ large enough, $\min\{a_{n_i}, b_{n_i}, c_{n_i}\}$ is increasing. The three steps together imply that $k_n<0$  for some $n$ large enough, which is a contradiction.
\medskip

Step I: We show that $k_n$ is decreasing in $n.$ There are the following two cases to verify. 

Case 1: $\sqrt{a_n}-\sqrt{b_n}-\sqrt{c_n}>0.$ In this case, by Lemma \ref{lambda2}, 
\begin{equation}\label{5.5}
\big(a_{n+1},b_{n+1},c_{n+1}\big)=\Big(b_n+c_n,\frac{a_nb_n}{b_n+c_n},\frac{a_nc_n}{b_n+c_n}\Big).
\end{equation}
Without loss of generality, we assume that $b_n>c_n.$ Then $b_{n+1}$ is the largest among $a_{n+1},$ $b_{n+1}$ and $c_{n+1}.$ By a direct calculation and that $\sqrt{a_n}>\sqrt{b_n}+\sqrt{c_n},$ we have
\begin{equation*}
\begin{split}
k_n-k_{n+1}&=\big(\sqrt{a_n}-\sqrt{b_n}-\sqrt{c_n}\big)-\big(\sqrt{b_{n+1}}-\sqrt{a_{n+1}}-\sqrt{c_{n+1}} \big)
>\frac{2c_n}{\sqrt{b_n+c_n}}>0.
\end{split}
\end{equation*}
Moreover, since $a_n+b_n+c_n=1$ and $a_n>0,$  we have $\sqrt{b_n+c_n}<1,$ and hence $k_n-k_{n+1}>2c_n.$ Therefore, we have
\begin{equation}\label{est}
k_n-k_{n+1}>2\min\{a_n,b_n,c_n\}.
\end{equation}   

Case 2: One of $\sqrt{b_n}-\sqrt{a_n}-\sqrt{c_n}$ and
$\sqrt{c_n}-\sqrt{a_n}-\sqrt{b_n}$ is strictly greater than $0.$ In this case, we without loss of generality assume that $\sqrt{b_n}-\sqrt{a_n}-\sqrt{c_n}>0.$  Then by Lemma \ref{lambda2},
\begin{equation}\label{5.7}
\begin{split}
\big(a_{n+1},&b_{n+1},c_{n+1}\big)=\\
&\bigg(\frac{a_nb_n}{a_nb_n+b_nc_n+(a_n-c_n)^2},\frac{(a_n-c_n)^2}{a_nb_n+b_nc_n+(a_n-c_n)^2},\frac{b_nc_n}{a_nb_n+b_nc_n+(a_n-c_n)^2}\bigg). 
\end{split}
\end{equation}
Without loss of generality, we assume that $a_n>c_n.$ Then $a_{n+1}$ is the largest among $a_{n+1},$ $b_{n+1}$ and $c_{n+1}.$ By a direct calculation and that $b_n=1-a_n-c_n,$ we have
\begin{equation*}
\begin{split}
\frac{k_{n+1}}{k_n}&=\frac{\sqrt{a_{n+1}}-\sqrt{b_{n+1}}-\sqrt{c_{n+1}}}{\sqrt{b_{n}}-\sqrt{a_{n}}-\sqrt{c_{n}}}=\sqrt{\frac{a_n+c_n-2\sqrt{a_nc_n}}{a_n+c_n-4a_nc_n}}.
\end{split}
\end{equation*}
From $\sqrt{b_n}>\sqrt{a_n}+\sqrt{c_n}$ and $a_n+b_n+c_n=1,$ we have  $a_n<\frac{1}{2},$ $c_n<\frac{1}{2},$ and hence $2\sqrt{a_nc_n}>4a_nc_n.$ As a consequence, ${k_{n+1}}/{k_n}<1.$
\medskip

Step II: We show that there is an infinite subsequence $\{\mathcal T_{n_i}\}$ of $\{\mathcal T_n\}$ such that $(a_{n_i},b_{n_i},c_{n_i})$ is in Case 1 of Step I. We use contradiction. For each $(a_n,b_n,c_n)$ in Case 2 of Step I, let $A_n=\max\{\lambda_n(y),\lambda_n(z)\}$ and let  $B_n=\min\{\lambda_n(y),\lambda_n(z)\}.$ Then $\sqrt{A_n}>\sqrt{B_n}+\sqrt{\lambda_n(x)}.$ By Lemma \ref{lambda2}, $(a_{n+1},b_{n+1},c_{n+1})$ is in Case 1 of Step I if and only if $\lambda_n(x)>B_n.$ Now suppose that there is an $m\in\mathbb N$ such that $(a_m,b_m,c_m)$ is in Case 2 of Step I and $B_n>\lambda_n(x)$ for all $n\geqslant m.$ Then by Lemma \ref{lambda2}, we have $\lambda_{n+1}(x)=\lambda_n(x)$ and 
$$\sqrt{B_{n+1}}=\frac{B_n-\lambda_n(x)}{\sqrt{A_n}}<\frac{B_n-\lambda_n(x)}{\sqrt{B_n}+\sqrt{\lambda_n(x)}}=\sqrt{B_n}-\sqrt{\lambda_n(x)}$$
 for $n\geqslant m.$ By induction, $\lambda_{n}(x)=\lambda_m(x)$ and $\sqrt{B_{n}}<\sqrt{B_m}-(n-m)\sqrt{\lambda_m(x)}$ for all $n>m,$ which is impossible.
\medskip

Step III: We show that for $i$ large enough, $\min\{a_{n_i}, b_{n_i}, c_{n_i}\}$ is increasing. In Figure \ref{Fig9} below, we let $\Delta=\big\{(a,b,c)\in\mathbb R_{>0}^3\ |\ a+b+c=1\big\},$ and for each $k$ let $C_k$ be the intersection of $\Delta$ with the set $\big\{(a,b,c)\in\mathbb R_{>0}^3\ |\ \max\{\sqrt{a}-\sqrt{b}-\sqrt{c},
\sqrt{b}-\sqrt{a}-\sqrt{c},
\sqrt{c}-\sqrt{a}-\sqrt{b}\}=k\big\}.$ A direct calculation shows that $C_k$'s are parts of the concentric circles centered at $(\frac{1}{3},\frac{1}{3},\frac{1}{3})$ with radii increasing in $k,$ and that $C_0$ is the inscribed circle of $\Delta.$ In Figure \ref{Fig9} (a), let $Q$ be the intersection of $\Delta$ and the set $\big\{(a,b,c)\in\mathbb R_{>0}^3\ |\ (b+c)^2=ac\big\}.$ Then $Q$ is a quadratic curve in $\Delta$ going through the points $(1,0,0)$ and $(\frac{1}{2},0,\frac{1}{2}).$  Let the line segment $P$ be the intersection of $\Delta$ and the plane $\big\{(a,b,c)\in\mathbb R^3\ |\ a=c\big\}.$ Then by (\ref{5.5}), if $(a_n,b_n,c_n)$ is on $Q$ with $b_n>c_n,$ then $(a_{n+1},b_{n+1},c_{n+1})$ is on $P.$ Denote by $H$ the line segment connecting $(\frac{1}{2},\frac{1}{2},0)$ and $(\frac{1}{2},0,\frac{1}{2}),$ and by $L$ the line segment connecting $(1,0,0)$ and $(0,1,0).$ Let $D$ be the region of $\Delta$ bounded by $Q,$ $H$ and $L,$ and let $E$ be the region in $\Delta$ bounded by $P,$ $H$ and $L.$  In Figure \ref{Fig9} (b), let $p$ be the intersection of $Q$ and $C_{k_0},$ let $\epsilon$ be the third coordinate of $p,$ let $L_{\epsilon}$ be the intersection of $\Delta$ and the plane $\{(a,b,c)\in\mathbb R^3\ |\ c=\epsilon\},$ and let $F$ be the region in $\Delta$ bounded by $C_{k_0},$ $L_{\epsilon},$ $H$ and $L.$ Note that $F$ is a subset of $D.$
 \begin{figure}[htbp]
\centering
\includegraphics[scale=0.45]{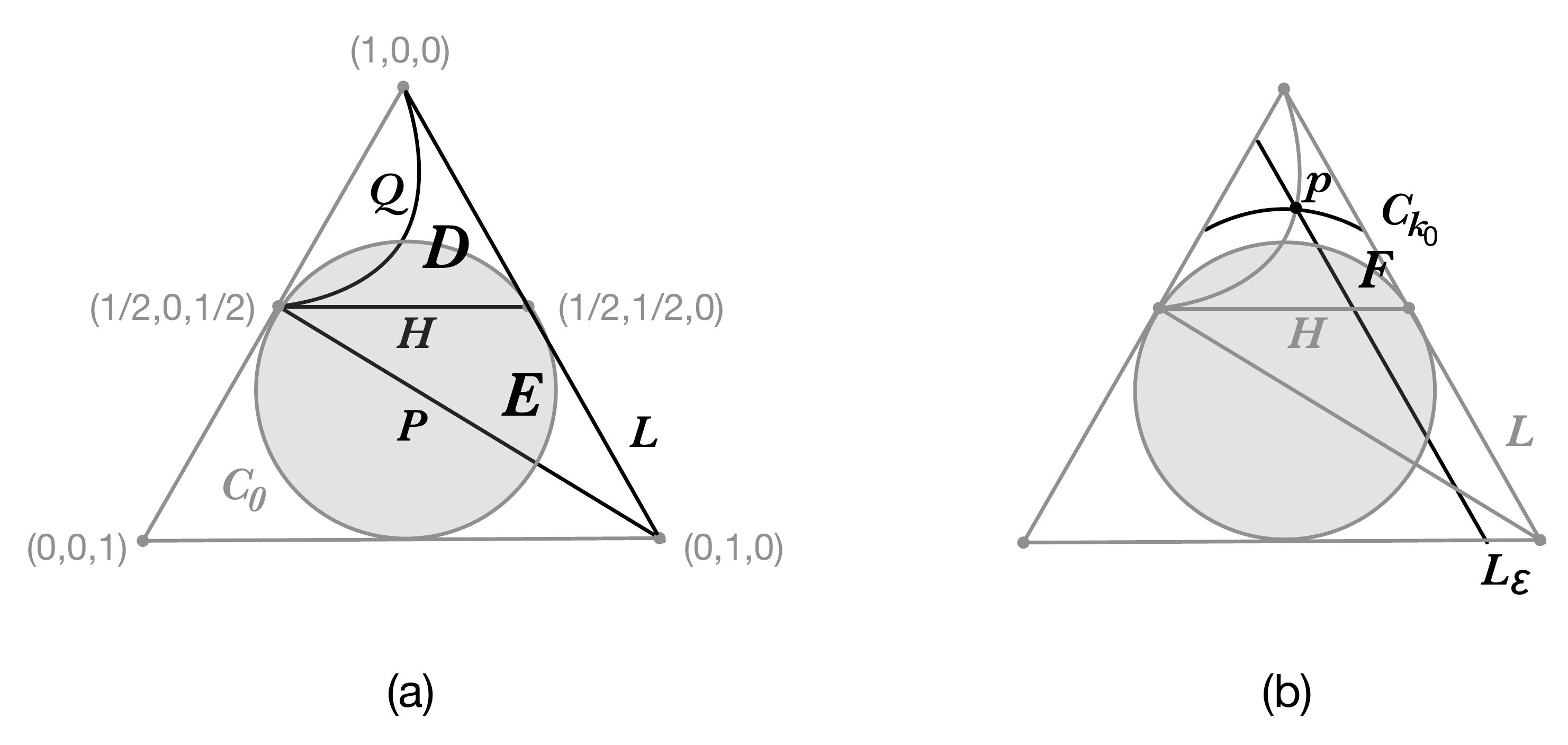}
\caption{}
\label{Fig9}
\end{figure}

Now consider the infinite subsequence $\{\mathcal T_{n_i}\}$ guaranteed by Step II such that $(a_{n_i},b_{n_i},c_{n_i})$ is in Case 1 of Step I. By (\ref{est}), there exists an $i_0$ such that $\min\{b_{n_i},c_{n_i}\}<\epsilon$ for all $i>i_0,$ since otherwise $k_{n_i+1}<0$ for $i$ large enough, and the algorithm stops. Without loss of generality, we assume that $b_{n_{i_0}}>c_{n_{i_0}},$ and we have the following two claims. 
\smallskip

Claim 1: If $(a_n,b_n,c_n)\in D$ and $b_n>c_n,$ then $(a_{n+1},b_{n+1},c_{n+1})\in E,$ $b_{n+1}>c_{n+1}$ and \begin{equation*}\label{5.8}
\min\{a_{n+1},b_{n+1},c_{n+1}\}>\min\{a_n,b_n,c_n\}.
\end{equation*}
Indeed, in this case, $c_n=\min\{a_n,b_n,c_n\}.$ By (\ref{5.5}), $(a_{n+1},b_{n+1},c_{n+1})\in E$ and $c_{n+1}>c_n.$ Furthermore, we have $b_{n+1}>c_{n+1},$ since otherwise $(a_{n+1},b_{n+1},c_{n+1})$ would be in the disk  bounded by the circle $C_0,$ i.e., $k_{n+1}<0$ and the algorithm stops. Therefore, $c_{n+1}=\min\{a_{n+1},b_{n+1},c_{n+1}\}$ and $\min\{a_{n+1},b_{n+1},c_{n+1}\}>\min\{a_n,b_n,c_n\}.$
\smallskip

Claim 2: For $n>n_0,$ if $(a_n,b_n,c_n)\in E$ and $b_n>c_n,$ then $(a_{n+1},b_{n+1},c_{n+1})\in D,$ $b_{n+1}>c_{n+1}$ and 
\begin{equation*}\label{5.9}
\min\{a_{n+1},b_{n+1},c_{n+1}\}>\min\{a_n,b_n,c_n\}.
\end{equation*}
Indeed, in this case, $c_n=\min\{a_n,b_n,c_n\}.$ By (\ref{5.7}), $(a_{n+1},b_{n+1},c_{n+1})$ is in the triangle above $H,$ $b_{n+1}>c_{n+1}$ and $c_{n+1}>c_n.$ Therefore, $c_{n+1}=\min\{a_{n+1},b_{n+1},c_{n+1}\}$ and
$\min\{a_{n+1},b_{n+1},c_{n+1}\}>\min\{a_n,b_n,c_n\}.$
Furthermore, since  $n>n_{i_0},$ we have $c_{n+1}<\epsilon,$ and by Step I, we have $k_{n+1}<k_0.$ As a consequence,  $(a_{n+1},b_{n+1},c_{n+1})\in F\subset D.$ Since the intersection $(\frac{2}{3}, \frac{1}{6}, \frac{1}{6})$ of the quadratic curve $Q$ and the circle $C_0$ lies on the line determined by $b=c,$ $F$ lies on the right half of $\Delta,$ and hence $b_{n+1}>c_{n+1}.$
\smallskip

Since $k_{n_{i_0}}<k_0$ and by assumption $b_{n_{i_0}}>c_{n_{i_0}}$ and $c_{n_{i_0}}<\epsilon,$  we have $(a_{n_{i_0}},b_{n_{i_0}},c_{n_{i_0}})\in F\subset D.$ 
By an induction and Claims 1 and 2, we have for all $m\geqslant 0$ that $(a_{n_{i_0}+2m},b_{n_{i_0}+2m},c_{n_{i_0}+2m})\in  D$ with $b_{n_{i_0}+2m}>c_{n_{i_0}+2m}$ and $(a_{n_{i_0}+2m+1},b_{n_{i_0}+2m+1},c_{n_{i_0}+2m+1})\in E$ with $b_{n_{i_0}+2m+1}>c_{n_{i_0}+2m+1},$  and hence for $n>n_{i_0}$ have $\min\{a_{n+1},b_{n+1},c_{n+1}\}>\min\{a_{n},b_{n},c_{n}\}.$   \end{proof}

Similar to the relative Euler class $\pm 1$ case, we have

\begin{proposition}\label{5.11} There are uncountably many $[\rho]\in\mathcal M_0(\Sigma_{0,4})$ such that all the tetrahedral triangulations of $\Sigma_{0,4}$ are $\rho$-admissible. For each such $\rho,$ there is a simple closed curve $\gamma$ on $\Sigma_{0,4}$ such that $\rho([\gamma])$ is an elliptic element in $PSL(2,\mathbb R).$
\end{proposition}

\begin{proof} Since the functions in Lemma \ref{lambda2} are rational in $\lambda(x),$ $\lambda(y)$ and $\lambda(z),$ the argument in the proof of Proposition \ref{5.5} applies here and proves the first part. The second part is a result of the Trace Reduction Algorithm. 
\end{proof}


\section{Connected components of $\mathcal M(\Sigma_{0,4})$}\label{Sec:6}

We describe the connected component of the character space $\mathcal M(\Sigma_{0,4})$ in this section. Recall that for a quadruple $s$ of positive and negative signs, $\mathcal M_k^s(\Sigma_{0,4})$ is the space of conjugacy classes of type-preserving representations with relative Euler class $k$ and signs of the punctures $s.$  Let $V=\{v_1,\dots,v_4\}$ be the set of punctures of $\Sigma_{0,4}.$ Then Theorem \ref{main3} is equivalent to the following

\begin{theorem}\label{main3'} 
\begin{enumerate}[(1)]
\item For $\{i,j,k,l\}=\{1,2,3,4\},$ let $s_{ij}\in \{\pm 1\}^V$ be defined by $s_{ij}(v_i)=s_{ij}(v_j)=-1$ and $s_{ij}(v_k)=s_{ij}(v_l)=+1.$ Then $$\mathcal M_0(\Sigma_{0,4})=\coprod_{\{i,j\}\subset\{1,\dots,4\}}\mathcal M_0^{s_{ij}}(\Sigma_{0,4}).$$

\item For $i\in\{1,\dots,4\},$ let $s_i\in \{\pm 1\}^V$ be defined by $s_i(v_i)=-1$ and $s_i(v_j)=+1$ for $j\neq i,$ and let $s_+\in \{\pm 1\}^V$ be defined by $s_+(v_i)=1$ for all $i\in\{1,\dots,4\}.$ Then
$$\mathcal M_1(\Sigma_{0,4})=\coprod_{i=1}^4\mathcal M_1^{s_{i}}(\Sigma_{0,4})\coprod\mathcal M_1^{s_{+}}(\Sigma_{0,4}).$$

\item For $i\in\{1,\dots,4\},$ let $s_{-i}\in \{\pm 1\}^V$ be defined by $s_{-i}(v_i)=+1$ and $s_{-i}(v_j)=-1$ for $j\neq i,$ and let $s_-\in \{\pm 1\}^V$ be defined by $s_-(v_i)=-1$ for all $i\in\{1,\dots,4\}.$ Then
$$\mathcal M_{-1}(\Sigma_{0,4})=\coprod_{i=1}^4\mathcal M_{-1}^{s_{-i}}(\Sigma_{0,4})\coprod\mathcal M_{-1}^{s_{-}}(\Sigma_{0,4}).$$

\item All the spaces $\mathcal M_0^{s_{ij}}(\Sigma_{0,4}),$  $\mathcal M_1^{s_{i}}(\Sigma_{0,4}),$  $\mathcal M_1^{s_{+}}(\Sigma_{0,4}),$  $\mathcal M_{-1}^{s_{-i}}(\Sigma_{0,4})$ and $\mathcal M_{-1}^{s_{-}}(\Sigma_{0,4})$ are connected.
\end{enumerate}
\end{theorem}

Let $\mathcal T$ be a tetrahedral triangulation of $\Sigma_{0,4}.$  Recall that $\mathcal M_{\mathcal T}(\Sigma_{0,4})$ is the space of conjugacy classes of type-preserving representations $\rho$ such that  $\mathcal T$ is $\rho$-admissible. By Theorem \ref{Kashaev1}, $\mathcal M_{\mathcal T}(\Sigma_{0,4})$ is a dense and open subset of $\mathcal M(\Sigma_{0,4}).$ 
Let $E$ and $T$ respectively be the sets of edges and ideal triangles of $\mathcal T,$ let $t_i\in T$ be the ideal triangle disjoint from the puncture $v_i$ and let $e_{ij}\in E$ be the edge connecting the punctures $v_i$ and $v_j.$
For $\lambda\in\mathbb R_{>0}^E,$ let $\lambda(x)=\lambda(e_{12})\lambda(e_{34}),$ $\lambda(y)=\lambda(e_{13})\lambda(e_{24})$ and $\lambda(z)=\lambda(e_{14})\lambda(e_{23}).$ We first show that the quantities $\lambda(x),$ $\lambda(y)$ and $\lambda(z)$ parametrize the components of $\mathcal M_{\mathcal T}(\Sigma_{0,4}).$

\begin{lemma}\label{delta} Let $\mathbb R_{>0}^E$ be with the principle $\mathbb R_{>0}^V$-bundle structure given by $(\mu\cdot\lambda)(e_{ij})=\mu(v_i)\lambda(e_{ij})\mu(v_j),$ and let $\mathbb R_{>0}^3$ be with the principle $\mathbb R_{>0}$-bundle structure defined by $r\cdot(a,b,c)=(ra,rb,rc).$ Then the map $\phi:\mathbb R_{>0}^E\rightarrow\mathbb R_{>0}^3$ sending $\big(\lambda(e_{12}),\dots,\lambda(e_{34})\big)$ to $\big(\lambda(x),\lambda(y),\lambda(z)\big)$ induces a diffeomorphism $\phi^*:\mathbb R_{>0}^E/\mathbb R_{>0}^V\rightarrow\mathbb R_{>0}^3/\mathbb R_{>0}.$
\end{lemma}

\begin{proof} Since $\phi(\mu\cdot\lambda)=\prod_{i=1}^4\mu(v_i)\cdot\phi(\lambda),$ $\phi^*$ is well defined, and since $\phi(a^{\frac{1}{2}}, b^{\frac{1}{2}}, c^{\frac{1}{2}}, c^{\frac{1}{2}}, b^{\frac{1}{2}},a^{\frac{1}{2}})=(a,b,c)$ for all $(a,b,c)\in\mathbb R_{>0}^3,$ $\phi^*$ is surjective. For the injectivity, we suppose that $\phi(\lambda')=r\cdot\phi(\lambda).$ Let $\nu_i(\lambda)=\prod_{j\neq i}\lambda(e_{ij})^2\prod_{j,k\neq i}\lambda(e_{jk})$ and let $\mu(v_i)=r^{\frac{1}{2}}\nu_i(\lambda')^{\frac{1}{6}}/\nu_i(\lambda)^{\frac{1}{6}}.$ Then $\lambda'(e_{ij})=\mu(v_i)\lambda(e_{ij})\mu(v_j).$ Therefore, $\phi^*$ is injective. The differentiability of $\phi^*$ and $(\phi^*)^{-1}$ follows from the definition of $\phi.$
\end{proof}

As a consequence of Theorem \ref{Kashaev2}, Lemma \ref{5.1}, Lemma \ref{57} and Lemma \ref{delta}, we have

\begin{corollary}\label{Delta}
Let $\mathcal T$ be a tetrahedral triangulation of $\Sigma_{0,4}$ with the set of ideal triangles $T.$  Then 
$$\mathcal M_{\mathcal T}(\Sigma_{0,4})\cong\coprod_{\epsilon\in\{\pm1\}^T}\Delta(\mathcal T,\epsilon),$$
where each $\Delta(\mathcal T,\epsilon)$ is a subset of $\Delta=\{(a,b,c)\in\mathbb R_{>0}^3\ |\ a+b+c=1\}$ defined as follows. 
\begin{enumerate}[(1)]

\item For $i\in\{1,\dots,4\},$ let $\epsilon_i\in\{\pm1\}^T$ be given by $\epsilon_i(t_i)=-1$ and $\epsilon_i(t_j)=1$ for $j\neq i,$ and let $\epsilon_{-i}\in\{\pm1\}^T$ be given by $\epsilon_{-i}(t_i)=1$ and $\epsilon_{-i}(t_j)=-1$ for $j\neq i.$ Then
$$\Delta(\mathcal T,\epsilon_{i})=\Delta(\mathcal T,\epsilon_{-i})=\{(a,b,c)\in\Delta\ |\ a\neq b+c,\ b\neq a+c\text{ and }c\neq a+b\}.$$

\item For $\{i,j,k,l\}=\{1,\dots,4\},$ let $\epsilon_{ij}\in\{\pm1\}^T$ be given by $\epsilon_{ij}(t_i)=\epsilon_{ij}(t_j)=-1$ and $\epsilon_{ij}(t_k)=\epsilon_{ij}(t_l)=1.$ Then
\begin{equation*}
\begin{split}
&\Delta(\mathcal T,\epsilon_{12})=\Delta(\mathcal T,\epsilon_{34})=\{(a,b,c)\in\Delta\ |\ a\neq b+c\},\\
&\Delta(\mathcal T,\epsilon_{13})=\Delta(\mathcal T,\epsilon_{24})=\{(a,b,c)\in\Delta\ |\ b\neq a+c\}\ \text{ and}\\
&\Delta(\mathcal T,\epsilon_{14})=\Delta(\mathcal T,\epsilon_{23})=\{(a,b,c)\in\Delta\ |\ c\neq a+b\}.
\end{split}
\end{equation*}
\end{enumerate}
\end{corollary}

\begin{proof}[Proof of Theorem \ref{main3'}] Suppose  $s$ is a quadruple of positive and negative signs and $k\in\{1,0,-1\}.$ Let $\mathcal T$ be a tetrahedral triangulation of $\Sigma_{0,4}$. Since $\mathcal M_{\mathcal T}(\Sigma_{0,4})$ is dense and open in $\mathcal M(\Sigma_{0,4}),$ $\mathcal M_k^s(\Sigma_{0,4})\neq\emptyset$ if and only if $\mathcal M_k^s(\Sigma_{0,4})\cap\mathcal M_{\mathcal T}(\Sigma_{0,4})\neq\emptyset.$  For (1), by Lemma \ref{57}, the only possibility for $\mathcal M_{\mathcal T}(\Sigma_{0,4})\cap \mathcal M_0^s(\Sigma_{0,4})\neq\emptyset$ is that $s=s_{ij}$ for some $\{i,j\}\subset\{1,\dots,4\}.$ For (2), by Lemma \ref{5.1}, the only possibility for $\mathcal M_{\mathcal T}(\Sigma_{0,4})\cap \mathcal M_1^s(\Sigma_{0,4})\neq\emptyset$ is that either $s=s_+,$ in which case $\lambda(x),$ $\lambda(y)$ and $\lambda(z)$ satisfy the triangular inequality, or $s=s_i$ for some $i\in\{1,\dots,4\},$ in which case $\lambda(x),$ $\lambda(y)$ and $\lambda(z)$ satisfy one of the inequalities in (\ref{anti-tri}). The proof of (3) is parallel to that of (2).

For (4), by symmetry, it suffices to prove the connectedness of $\mathcal M_0^{s_{12}}(\Sigma_{0,4}),$  $\mathcal M_1^{s_{1}}(\Sigma_{0,4})$ and $\mathcal M_1^{s_{+}}(\Sigma_{0,4}).$ We consider the following subsets of $\Delta.$ Let $\Delta_x(\mathcal T,\epsilon)$ (resp. $\Delta_y(\mathcal T,\epsilon)$ and $\Delta_z(\mathcal T,\epsilon)$) be the set of points $(a,b,c)\in\Delta$ such that $a>b+c$ (resp. $b>a+c$ and $c>a+b$), let $\Delta_x^c(\mathcal T,\epsilon)$ (resp. $\Delta_y^c(\mathcal T,\epsilon)$ and $\Delta_z^c(\mathcal T,\epsilon)$) be the set of points  $(a,b,c)\in\Delta$ such that $a<b+c$ (resp. $b<a+c$ and $c<a+b$) and let $\Delta^c(\mathcal T,\epsilon)=\Delta_x^c(\mathcal T,\epsilon)\cap\Delta_y^c(\mathcal T,\epsilon)\cap\Delta_z^c(\mathcal T,\epsilon).$ See Figure \ref{Fig10}. By Theorem \ref{Kashaev2}, Lemma \ref{5.1}, Lemma \ref{57} and Corollary \ref{Delta}, we have via the lengths coordinate that
\begin{figure}[htbp]
\centering
\includegraphics[scale=0.4]{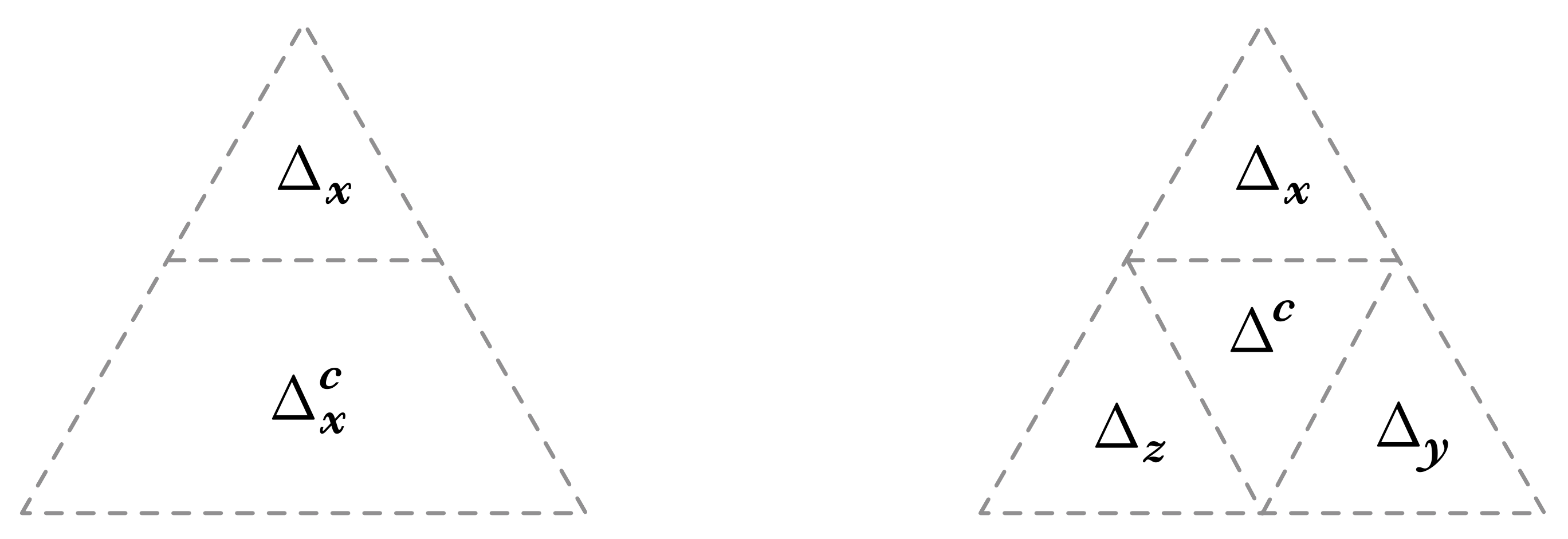}
\caption{}
\label{Fig10}
\end{figure}
\begin{enumerate}[(1)]
\item $\Delta_x(\mathcal T,\epsilon_{12})\coprod\Delta_x^c(\mathcal T,\epsilon_{34})$ is diffeomorphic to a dense and open subset of $\mathcal M_0^{s_{12}}(\Sigma_{0,4}),$
\item $\Delta_x(\mathcal T,\epsilon_2)\coprod\Delta_y(\mathcal T,\epsilon_3)\coprod\Delta_z(\mathcal T,\epsilon_4)$ is diffeomorphic to a dense and open subset of $\mathcal M_1^{s_1}(\Sigma_{0,4}),$ and
\item $\coprod_{i=1}^4\Delta^c(\mathcal T,\epsilon_i)$ is diffeomorphic to a dense and open subset of $\mathcal M_1^{s_+}(\Sigma_{0,4}).$ 
\end{enumerate}
In the rest of the proof, we let $\mathcal T'$ be the tetrahedral triangulation of $\Sigma_{0,4}$ obtained from $\mathcal T$ by doing a simultaneous diagonal switch $S_z,$ let $T'
$ be the set of ideal triangles of $\mathcal T',$ and let $\epsilon'_i$ and $\epsilon'_{ij}\in\{\pm1\}^{T'}$ be sign assignments defined in the same way as $\epsilon_i$ and $\epsilon_{ij}.$ For $(a,b,c)\in\mathbb R_{>0}^3,$ we let $[a,b,c]\doteq(\frac{a}{a+b+c},\frac{b}{a+b+c},\frac{c}{a+b+c})\in\Delta.$ 

For the connectedness of $\mathcal M_0^{s_{12}}(\Sigma_{0,4}),$ since both $\Delta_x(\mathcal T,\epsilon_{12})$ and $\Delta_x^c(\mathcal T,\epsilon_{34})$ are connected, it suffices to choose two points $p$ and $q$ respectively in $\Delta_x(\mathcal T,\epsilon_{12})$ and $\Delta_x^c(\mathcal T,\epsilon_{34})$ and find a path in $\mathcal M_0^{s_{12}}(\Sigma_{0,4})$ connecting $p$ and $q.$ Now let $p=(a,b,c)\in\Delta_x(\mathcal T,\epsilon_{12})$ and let $q=(a',b',c')\in\Delta_x^c(\mathcal T,\epsilon_{34})$ with $a'>b'.$ By Proposition \ref{change} and Lemma \ref{lambda2}, $p$ corresponds to the point $p'=[a,b,(a-b)^2/c]\in\Delta_x(\mathcal T',\epsilon'_{12})$ and $q$ corresponds to the point $q'=[a',b',(a'-b')^2/c']\in\Delta_x(\mathcal T',\epsilon'_{12}).$  Since $\Delta_x(\mathcal T',\epsilon'_{12})$ is connected, there is a path in $\Delta_x(\mathcal T',\epsilon'_{12})$ connecting $p'$ and $q',$ giving a path in $\mathcal M_0^{s_{12}}(\Sigma_{0,4})$ connecting $p$ and $q.$ 

For the connectedness of $\mathcal M_1^{s_1}(\Sigma_{0,4}),$ we let $p=(a,b,c)\in\Delta_x(\mathcal T,\epsilon_2)$ and let $q=(a',b',c')\in\Delta_y(\mathcal T,\epsilon_3).$ By Proposition \ref{change}, Lemma \ref{lambda'} and Lemma \ref{preserve}, $p$ corresponds to the point $p'=[a,b,|a^2-b^2|/c]\in\Delta_z(\mathcal T',\epsilon'_{4})$ and $q$ corresponds to the point $q'=[a',b',|a'^2-b'^2|/c']\in\Delta_z(\mathcal T',\epsilon'_{4}).$  Since $\Delta_z(\mathcal T',\epsilon'_{4})$ is connected, there is a path in $\Delta_z(\mathcal T',\epsilon'_{4})$ connecting $p'$ and $q',$ giving a path in $\mathcal M_1^{s_{1}}(\Sigma_{0,4})$ connecting $p$ and $q.$ Similarly, any pair of points $q\in\Delta_y(\mathcal T,\epsilon_3)$ and $r\in\Delta_z(\mathcal T,\epsilon_4)$ and any pair of points $p\in\Delta_y(\mathcal T,\epsilon_2)$ and $r\in\Delta_z(\mathcal T,\epsilon_4)$ can respectively be connected by paths in $\mathcal M_1^{s_1}(\Sigma_{0,4}).$ Therefore, $\mathcal M_1^{s_1}(\Sigma_{0,4})$ is connected. 

Finally, for the connectedness of $\mathcal M_1^{s_+}(\Sigma_{0,4}),$ we let $p=(a,b,c)\in\Delta^c(\mathcal T,\epsilon_2)$ with $a>b$ and let $q=(a',b',c')\in\Delta^c(\mathcal T,\epsilon_3)$ with $b'>a'.$ By Proposition \ref{change}, Lemma \ref{lambda'} and Lemma \ref{preserve}, $p$ corresponds to the point $p'=[a,b,|a^2-b^2|/c]\in\Delta^c(\mathcal T',\epsilon'_{4})$ and $q$ corresponds to the point $q'=[a',b',|a'^2-b'^2|/c']\in\Delta^c(\mathcal T',\epsilon'_{4}).$  Since $\Delta^c(\mathcal T',\epsilon'_{4})$ is connected, there is a path in $\Delta^c(\mathcal T',\epsilon'_{4})$ connecting $p'$ and $q',$ giving a path in $\mathcal M_1^{s_{+}}(\Sigma_{0,4})$ connecting $p$ and $q.$ Similarly, all the other pieces can be connected by paths in $\mathcal M_1^{s_+}(\Sigma_{0,4}),$ and $\mathcal M_1^{s_+}(\Sigma_{0,4})$ is connected.
\end{proof}


\section{Ergodicity of the $Mod(\Sigma_{0,4})$-action}\label{Sec:7}


The goal of this section is to prove the ergodicity of the $Mod(\Sigma_{0,4})$-action on the non-external connected components of $\mathcal M(\Sigma_{0,4}).$ To use the techniques we used in the previous sections, we need to understand the measure on $\mathcal M(\Sigma_{0,4})$ induced by the Goldman symplectic $2$-form in terms the quantities $\lambda(x),$ $\lambda(y)$ and $\lambda(z).$ Let $\mathcal T$ be a tetrahedral triangulation of $\Sigma_{0,4},$ and let $T$ be the set of ideal triangles of $\mathcal T.$ For each $\epsilon\in\{\pm1\}^T,$ let $\Delta(\mathcal T,\epsilon)$ be the subset of $\mathbb R_{>0}^3$ defined in Corollary \ref{Delta}. Then by Equation (\ref{WP}), Lemma \ref{delta}, Corollary \ref{Delta} and a direct calculation,  we have the following

\begin{proposition}\label{WeilP} 
For each $\epsilon\in\{\pm1\}^T,$ the $2$-form
$$\omega=\frac{d\lambda(x)\wedge d\lambda(y)}{\lambda(x)\lambda(y)}+\frac{d\lambda(y)\wedge d\lambda(z)}{\lambda(y)\lambda(z)}+\frac{d\lambda(z)\wedge d\lambda(x)}{\lambda(z)\lambda(x)}$$
on $\Delta(\mathcal T,\epsilon)$ corresponds to the Goldman symplectic $2$-form $\omega_{WP}$ on $\mathcal M(\Sigma_{0,4}),$ 
and the measure induced by $\omega$ is in the same measure class of the Lebesgue measure on $\Delta(\mathcal T,\epsilon).$ 
\end{proposition}

As a consequence, we have

\begin{proposition}\label{dense} For $k\in\{-1,0,1\},$ the set $\Omega_k(\Sigma_{0,4})$ consisting of conjugacy classes of type-preserving representations $\rho$ with the relative Euler class $e(\rho)=k$ such that all the tetrahedra triangulation of $\Sigma_{0,4}$ are $\rho$-admissible is a full measure subset of $\mathcal M_k(\Sigma_{0,4}),$ and is invariant under the $Mod(\Sigma_{0,4})$-action.
\end{proposition}

\begin{proof}
By the proof of Proposition \ref{5.5}, $\mathcal M_1(\Sigma_{0,4})\setminus \Omega_1(\Sigma_{0,4})$ is a countable union of Lebesgue measure zero subsets, hence is of Lebesgue measure zero. Then by Proposition \ref{WeilP}, $\mathcal M_1(\Sigma_{0,4})\setminus \Omega_1(\Sigma_{0,4})$ is a null set in the measure induced by the Goldman symplectic $2$-form. By the similar argument, $\Omega_0(\Sigma_{0,4})$ and $\Omega_{-1}(\Sigma_{0,4})$ are respectively full measure subsets of $\mathcal M_{0}(\Sigma_{0,4})$ and $\mathcal M_{-1}(\Sigma_{0,4}).$ By Proposition \ref{evenodd}, since all the simultaneous diagonal switches act on each $\Omega_k(\Sigma_{0,4}),$ so does $Mod(\Sigma_{0,4}).$
\end{proof}

\begin{remark}
Since $\Omega_{1}(\Sigma_{0,4})$ is dense in $\mathcal M_{1}(\Sigma_{0,4})$ and a representation in $\Omega_{1}(\Sigma_{0,4})\cap\coprod_{i=1}^4\mathcal M_{1}^{s_{i}}(\Sigma_{0,4})$ sends each simple closed curve to a hyperbolic element, by continuity, every representation in $\coprod_{i=1}^4\mathcal M_{1}^{s_{i}}(\Sigma_{0,4})$ sends each simple closed curve to either a hyperbolic or a parabolic element. In \cite{D}, Delgado explicitly constructed a family $\{\rho_t\}$ of representations  in $\mathcal M_1(\Sigma_{0,4})$ that send every simple closed curve to either a hyperbolic or a parabolic element, and for each $\rho_t$, at least one simple closed curve is sent to a parabolic element. Therefore, the representations $\{\rho_t\}$ are in the measure zero subset $\mathcal M_{1}(\Sigma_{0,4})\setminus \Omega_{1}(\Sigma_{0,4}).$ 
\end{remark}

Let $V=\{v_1,\dots,v_4\}$ be the set of punctures of $\Sigma_{0,4}.$ For $k\in \{-1,0,1\}$ and $s\in\{\pm 1\}^V,$ let $\Omega_k^s(\Sigma_{0,4})=\Omega_k(\Sigma_{0,4})\cap \mathcal M_k^s(\Sigma_{0,4}).$ By Theorem \ref{main3'} and Proposition \ref{dense}, Theorem \ref{main4} follows from the following

\begin{theorem}\label{E}
\begin{enumerate}[(1)]
\item The $Mod(\Sigma_{0,4})$-action on $\Omega_0^{s_{ij}}(\Sigma_{0,4})$ is ergodic for each $\{i,j\}\subset\{1,\dots, 4\}.$

\item The $Mod(\Sigma_{0,4})$-action on $\Omega_{1}^{s_+}(\Sigma_{0,4})$ and $\Omega_{-1}^{s_-}(\Sigma_{0,4})$ is ergodic.

\item The $Mod(\Sigma_{0,4})$-action on $\Omega_{1}^{s_{ i}}(\Sigma_{0,4})$ and $\Omega_{-1}^{s_{-i}}(\Sigma_{0,4})$ is ergodic for each $i\in\{1,\dots, 4\}.$
\end{enumerate}
\end{theorem}

\begin{remark} The ergodicity of the $Mod(\Sigma_{0,4})$-action on the components of $\mathcal M_{\pm1}(\Sigma_{0,4})$ was first known to Maloni-Palesi-Tan in \cite{MPT} using the Markoff triple technique. 
\end{remark}

\subsection{A proof of Theorem \ref{E} (1)}


By symmetry, it suffices to prove the ergodicity of the $Mod(\Sigma_{0,4})$-action on $\Omega_0^{s_{12}}(\Sigma_{0,4}).$ Let $\Delta=\{(a,b,c)\in\mathbb R_{>0}^3\ | a+b+c=1\},$ and let
$\Delta_x=\{(a,b,c)\in\Delta\ |\ a\neq b+c\}.$ By Theorem \ref{Kashaev2},  Lemma \ref{delta} and Corollary \ref{Delta}, given a tetrahedral triangulation of $\Sigma_{0,4},$ $\Delta_x$ is diffeomorphic to a dense and open subset of $\mathcal M_0^{s_{12}}(\Sigma_{0,4}),$ where the diffeomorphism is given by the quantities $\lambda(x),$ $\lambda(y)$ and $\lambda(z).$ 

Consider the embedding of $i:\Delta\rightarrow\mathbb R_{>0}^3$ defined by $i((a,b,c))=(1,b/a, c/a).$ Then 
$$i(\Delta_x)=\{(1,b,c)\in \mathbb R_{>0}^3\ |\ b+c\neq 1\}.$$
Let $\Omega_X$ be the subset of $i(\Delta_x)$ consisting of the elements coming from $\Omega_0^{s_{12}}(\Sigma_{0,4}).$ As an immediate consequences of Lemma \ref{lambda2}, the simultaneous diagonal switches $S_y$ and $S_z$ act on $\Omega_X$ by
$$S_y\big((1,b,c)\big)=\Big(1,\frac{(1-c)^2}{b},c\Big)$$
and
$$S_z\big((1,b,c)\big)=\Big(1,b,\frac{(1-b)^2}{c}\Big).$$

\begin{lemma}\label{MM}
Let $\langle D_X\rangle$ be the cyclic subgroup  of $Mod(\Sigma_{0,4})$ generated by the Dehn twist $D_X$ along the distinguished simple closed curve $X.$ Then
for every $k\in(-2,2),$ the ellipse 
$$E_k=\{(1,b,c)\in \Omega_X\ |\ (b+c-1)^2=(k+2)bc\}$$
 is invariant under the action of  $\langle D_X\rangle,$ and
 for almost every $k\in(-2,2),$ the action of $\langle D_X\rangle$ on $E_k$ is ergodic.
 \end{lemma}

\begin{proof} 

A direct calculation shows that $E_k$ is invariant under the actions of $S_y$ and $S_z$ for all $k\in(-2,2).$ Recall that $D_X=S_zS_y.$ Therefore,  the ellipse $E_k$ is invariant under the $\langle D_X\rangle$-action. By lemma \ref{lambda2}, the action of $S_y$ and $S_z$ respectively move a point $p$ on $E_k$ vertically and horizontally.  As show in Figure \ref{Fig11}, the an affine transformation of $\mathbb R^2$ sending the ellipse $E_k$ to a circle $C_k$ sends the vertical and the horizontal lines in $\mathbb R^2$ respectively to two family of parallel lines in $C_k.$ As a consequence, for each point $p$ on $C_k,$ the angle $\angle pp'D_X(p)$ is a constant $\theta_k/2$ depending only on $k,$
and the center angle $\angle pOD_X(p)=2\angle pp'D_X(p)=\theta_k.$ Therefore, $D_X$ acts on $C_k$ by a rotation of angle $\theta_k.$ Since $\theta_k$ is an irrational multiple of $2\pi$ for almost every $k\in(-2,2),$ the action of $\langle D_X \rangle$ is ergodic.

\begin{figure}[htbp]
\centering
\includegraphics[scale=0.4]{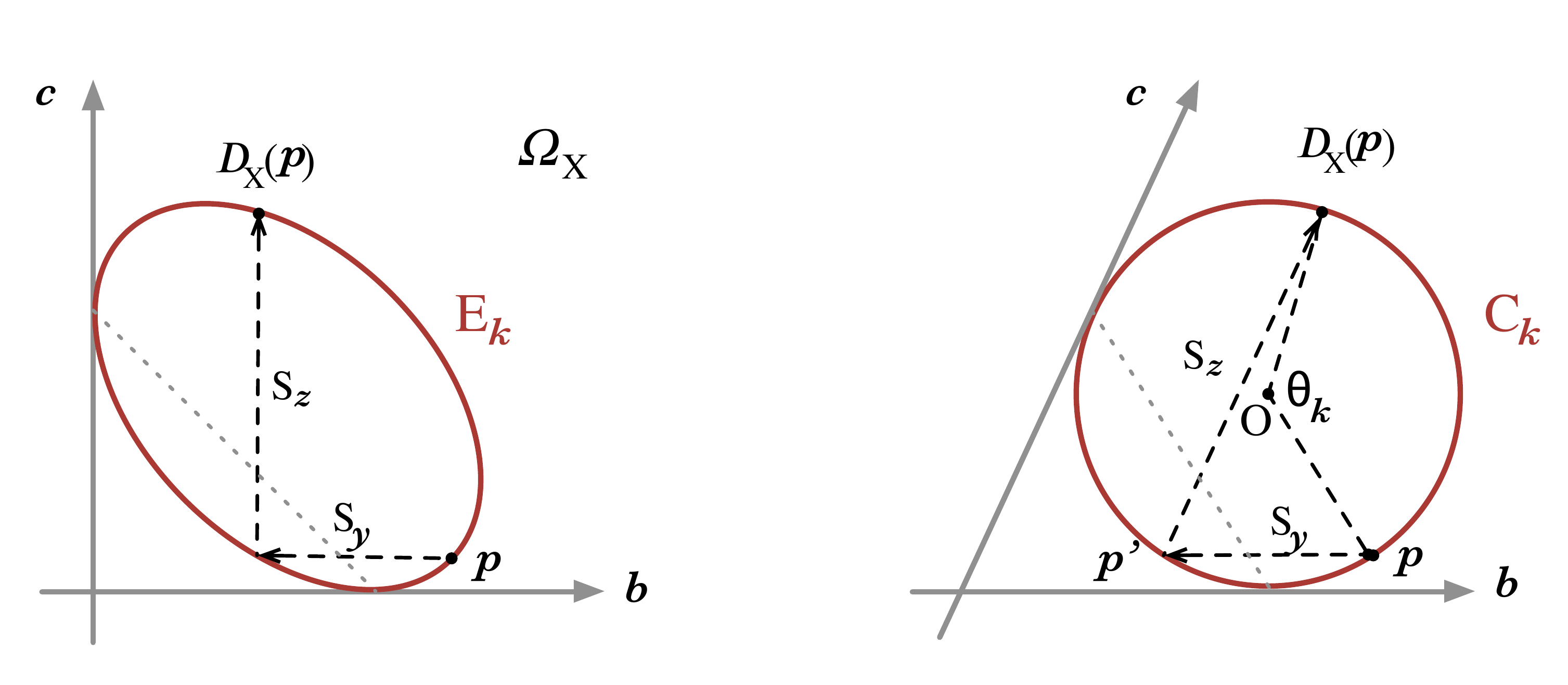}
\caption{}
\label{Fig11}
\end{figure}
\end{proof}

\begin{lemma}\label{MM2}
Let $D_YD_Z$ be the self-diffeomorphism of $\Sigma_{0,4}$ given by the Dehn twist $D_Z$ along the distinguished curve $Z$ in $\mathcal T$ followed by the Dehn twist $D_Y$ along the distinguished curve $Y$ in $D_Z(\mathcal T),$ and let $\langle D_YD_Z\rangle$ be the cyclic subgroup  of $Mod(\Sigma_{0,4})$ generated by $D_YD_Z.$  Then for every $k\in (-2,2),$ the quartic curve
$$Q_k=\{(1,b,c)\in \Omega_X\ |\ (b+c)^2(b+c-1)^2=(k+2)bc\}$$ is invariant under the action of $\langle D_YD_Z\rangle,$ and for almost every $k\in(-2,2),$ the action of $\langle D_YD_Z\rangle$ on $Q_k$ is ergodic. 
\end{lemma}

\begin{proof} A direct calculation shows that $Q_k$ is the $S_x$-image of $E_k.$ Since $D_YD_Z=S_xS_zS_yS_x=S_xD_XS_x,$ the map $S_z:E_k\rightarrow Q_k$ is $\mathbb Z$-equivariant, where $1\in\mathbb Z$ acts on $E_k$ by $D_X$ and acts on $Q_k$ by $D_YD_Z.$ By Lemma \ref{MM},  $\langle D_YD_Z\rangle$ acts on $Q_k$ for every $k\in(-2,2),$ and the action is ergodic for almost every $k\in(-2,2).$
\end{proof}

\begin{proof}[Proof of Theorem \ref{E} (1)] 
We show that every $Mod(\Sigma_{0,4})$-invariant measurable function $F:\Omega_X\rightarrow\mathbb R$ is almost everywhere a constant. Consider the following region 
$$R=\{(1,b,c)\in \Omega_X\ |\ (b+c-1)^2<4bc,\ b+c<1\}$$
in $\Omega_X$ inclosed by the parabola 
$P=\{(1,b,c)\in \Omega_X\ |\ (b+c-1)^2=4bc\}$ and the line segment $L=\{(1,b,c)\in \Omega_X\ |\ b+c=1\}.$ We claim that each point $p=(1,b_0,c_0)$ in $R$ is an intersection an ellipse $E_{k_1}$ and a quantic curve $Q_{k_2}$ for some $k_1,$ $k_2\in(-2,2).$ Indeed, we can let $k_1=\frac{(b_0+c_0-1)^2}{b_0c_0}-2$ and let $k_2=\frac{(b_0+c_0)^2(b_0+c_0-1)^2}{b_0c_0}-2.$ Since $(b_0+c_0-1)^2<4b_0c_0,$ $k_1\in(-2,2),$ and since $b_0+c_0<1,$ $k_2\in(-2,2).$ A direct calculation shows that the intersection of $E_{k_1}$ and $Q_{k_2}$ is transverse at $p,$ i.e., the gradients $\nabla E_{k_1}(p)$ and $\nabla Q_{k_2}(p)$ span the tangent space of $\Omega_X$ at $p.$ Then by Lemma \ref{MM} and Lemma \ref{MM2}, the restriction of $F$ to $R$ is almost everywhere a constant. For $p\in \Omega_X,$ let $O(p)$ be the $Mod(\Sigma_{0,4})$-orbit of $p.$ To show that the $F$ is almost everywhere a constant in $\Omega_X,$ it suffices to show that $O(p)\cap R\neq\emptyset$ for almost every $p$ in $\Omega_X.$ Let $R'$ be the region in $\Omega_X$ enclosed by parabola $P,$ i.e., $R'=\{(1,b,c)\in \Omega_X\ |\ (b+c-1)^2<4bc\}.$ Then $R'$ is foliated by the ellipses $\{E_k\}.$ We note that the parabola $P$ is the $i$-image of the inscribe circle $C_0$ of $\Delta.$ Then by the Trace Reduction Algorithm, Lemma \ref{finite} and Proposition \ref{dense}, for almost every $p$ in $\Omega_X,$ there is a composition $\phi$ of finitely many, say $m,$ simultaneous diagonal switches such that $\phi(p)\in R'.$ By Proposition \ref{evenodd}, if $m$ is even, then $\phi\in Mod(\Sigma_{0,4})$ and $O(p)\cap R'\neq\emptyset;$ and if $m$ is odd, then $\phi'=S_y\phi\in Mod(\Sigma_{0,4}).$ Since $S_y$ keeps invariant an ellipse $E_k\subset R'$ passing through $\phi(p),$ $\phi'(p)=S_y\phi(p)\in E_k\subset R',$ and hence $O(p)\cap R'\neq\emptyset.$ Finally, by Lemma \ref{MM}, for almost every $p$ in $R',$ there is $n$ such that $D_X^n(p)\in E_k\cap R\subset R.$\end{proof}


\subsection{A proof of Theorem \ref{E} (2)}

By symmetry, it suffices to prove the ergodicity of the $Mod(\Sigma_{0,4})$-action on $\Omega_1^{s_+}.$ The strategy is to find two transversely intersecting families of curves $\{E_{X,k}\}$ and $\{E_{Y,k}\}$ foliating $\mathcal M_1^{s_+}(\Sigma_{0,4})$ such that the $\langle D_X\rangle$ -action on almost every $E_{X,k}$ and the $\langle D_Y\rangle$-action on almost every $E_{Y,k}$ is ergodic. Let $\mathcal T$ be an tetrahedral triangulation of $\Sigma_{0,4},$ and let $T$ be the set of ideal triangles of $\mathcal T.$ Let $\Delta=\{(a,b,c)\in\mathbb R_{>0}^3\ | a+b+c=1\},$ and for $\epsilon\in\{\pm 1\}^T,$  let $\Delta^c(\mathcal T,\epsilon)=\{(a,b,c)\in\Delta\ |\ a<b+c, b<a+c,c<b+a\}.$
By Theorem \ref{Kashaev2}, Lemma \ref{5.1} and Corollary \ref{Delta}, $\coprod_{i=1}^4\Delta^c(\mathcal T,\epsilon_i)$ is diffeomorphic to a dense and open subset of $\mathcal M_1^{s_+}(\Sigma_{0,4}),$ where the diffeomorphism is given by the quantities $\lambda(x),$ $\lambda(y)$ and $\lambda(z).$ 

 We define the embedding $i_X:\coprod_{i=1}^4\Delta^c(\mathcal T,\epsilon_i)\rightarrow \mathbb R^3$ by
\begin{equation*}
i_X((a,b,c))=\begin{cases}
(1,b/a,c/a), & \text{if }(a,b,c)\in \Delta^c(\mathcal T,\epsilon_1)\\
(1,-b/a,-c/a), & \text{if }(a,b,c)\in \Delta^c(\mathcal T,\epsilon_2)\\
(1,-b/a,c/a), & \text{if }(a,b,c)\in \Delta^c(\mathcal T,\epsilon_3)\\
(1,b/a,-c/a), & \text{if }(a,b,c)\in \Delta^c(\mathcal T,\epsilon_4).\\
\end{cases}
\end{equation*}
For $i\in\{1,\dots, 4\},$ we let $\Omega_{X,i}$ be the subset of $i_X(\Delta^c(\mathcal T,\epsilon_i))$ consisting of the elements coming from $\Omega_1^{s_+}(\Sigma_{0,4}),$ and let $\Omega_X=\coprod_{i=1}^4\Omega_{X,i}.$ (See Figure \ref{Fig12}.) 

\begin{figure}[htbp]
\centering
\includegraphics[scale=0.5]{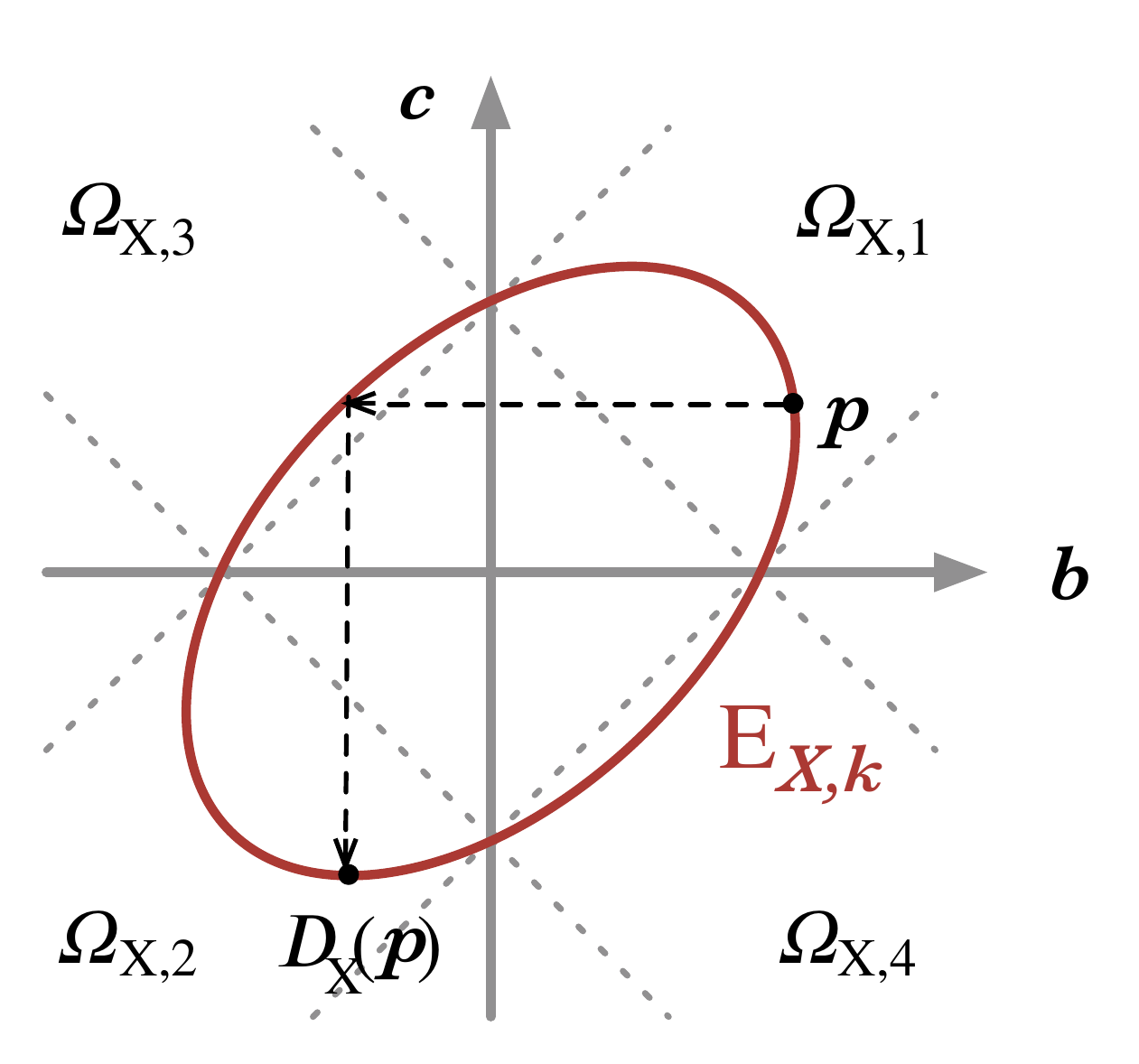}
\caption{}
\label{Fig12}
\end{figure}

\begin{lemma}\label{MM3}
For every $k\in(-2,2),$ the ellipse 
$$E_{X,k}=\{(1,b,c)\in \Omega_X\ |\ b^2+c^2-1=kbc\}$$
 is invariant under the action of  $\langle D_X\rangle,$ and for almost every $k\in(-2,2),$ the action of $\langle D_X\rangle$ on $E_{X,k}$ is ergodic.
 \end{lemma}
 
\begin{proof} For $(x,y,z)\in \mathbb R^3,$ let $|(a,b,c)|=(|a|, |b|, |c|).$ By Lemma \ref{lambda'}, the action of $S_y$ and $S_z$ on $\Omega_X$ satisfies 

$$\big|S_y\big((1,b,c)\big)\big|=\Big(1,\Big|\frac{c^2-1}{b}\Big|, |c|\Big)$$
and
$$\big|S_z\big((1,b,c)\big)\big|=\Big(1,|b|,\Big|\frac{b^2-1}{c}\Big|\Big).$$
Therefore, we have
$$\big| D_X\big((1,b,c)\big)\big|=\Big(1,\Big|\frac{c^2-1}{b}\Big|, \Big|\frac{(\frac{c^2-1}{b}\big)^2-1}{c}\Big|\Big).$$
We claim that 
\begin{equation}\label{MM4}
D_X\big((1,b,c)\big)=\Big(1, \frac{c^2-1}{b}, \frac{(\frac{c^2-1}{b}\big)^2-1}{c}\Big).
\end{equation}
If (\ref{MM4}) is true, then a direct calculation shows that $D_X((1,b,c))$ is on $E_{X,k}.$ 

To verify (\ref{MM4}), we let $\mathcal T'$ be the tetrahedral triangulation obtained from $\mathcal T$ by doing $S_y,$ and let $\mathcal T''$ be the tetrahedral triangulation obtained from $\mathcal T'$ by doing $S_z.$ Let $\epsilon'$ and $\epsilon''$ respectively be the signs of $\rho$ assigned to the ideal triangles of $\mathcal T'$ and $\mathcal T''.$  In $\mathcal T,$ $\mathcal T'$ and $\mathcal T'',$ we denote uniformly by  $t_i$ the ideal triangle disjoint from the puncture $v_i.$ If $p=(1,b,c)\in \Omega_{X,1},$ i.e. $i_X^{-1}(p)\in\Delta^c(\mathcal T,\epsilon_1),$ then we consider the following cases.

Case 1: $c>1$ and $\frac{c^2-1}{b}>1.$ In this case, since $\epsilon_1(t_{1})=-1$ and $c>1,$ we have by Proposition \ref{change} that $\epsilon'(t_{2})=-1.$ Since $\frac{c^2-1}{b}>1,$ by Proposition \ref{change} again, $\epsilon''(t_{1})=-1.$ Therefore, $D_X(i_X^{-1}(p))\in \Delta^c(\mathcal T'', \epsilon_1),$ and (\ref{MM4}) follows.

Case 2: $c>1$ and $\frac{c^2-1}{b}<1.$ In this case, by Proposition \ref{change}, $\epsilon'(t_{2})=-1$ and $\epsilon''(t_{4})=-1.$ Therefore, $D_X(i_X^{-1}(p))\in \Delta^c(\mathcal T'', \epsilon_4),$ and 
and (\ref{MM4}) follows.

Case 3: $c<1$ and $\frac{c^2-1}{b}>1.$  In this case, by Proposition \ref{change}, $\epsilon'(t_{4})=-1$ and $\epsilon''(t_{3})=-1.$ Therefore, $D_X(i_X^{-1}(p))\in \Delta^c(\mathcal T'', \epsilon_3),$ and 
and (\ref{MM4}) follows.

Case 4: $c<1$ and $\frac{c^2-1}{b}<1.$ In this case, by Proposition \ref{change}, $\epsilon'(t_{4})=-1$ and $\epsilon''(t_{2})=-1.$ Therefore, $D_X(i_X^{-1}(p))\in \Delta^c(\mathcal T'', \epsilon_2),$ and 
and (\ref{MM4}) follows. 

The verification of (\ref{MM4}) for $p$ in $\Omega_{X,2},$ $\Omega_{X,3}$ and $\Omega_{X,4}$ is similar, and is left to the readers. 

By (\ref{MM4}), the action of $D_X$ on $E_{X,k}$ is a horizontal translation followed by a vertical translation. See Figure \ref{Fig12}. By doing a suitable affine transform, the $\langle D_X\rangle$-action is a rotation of an angle $\theta_k$ on a circle, where $\theta_k$ is an irrational multiple of $2\pi$ for almost every $k.$ Therefore, for almost every $k\in(-2,2),$ the $\langle D_X\rangle$-action on $E_{X,k}$ is ergodic.
\end{proof}

Consider the embedding $i_Y:\coprod_{i=1}^4\Delta^c(\mathcal T,\epsilon_i)\rightarrow \mathbb R^3$ by
\begin{equation*}
i_Y((a,b,c))=\begin{cases}
(a/b,1,c/b), & \text{if }(a,b,c)\in \Delta^c(\mathcal T,\epsilon_1)\\
(-a/b,1,c/b), & \text{if }(a,b,c)\in \Delta^c(\mathcal T,\epsilon_2)\\
(-a/b,1,-c/b), & \text{if }(a,b,c)\in \Delta^c(\mathcal T,\epsilon_3)\\
(a/b,1,-c/b), & \text{if }(a,b,c)\in \Delta^c(\mathcal T,\epsilon_4)\\
\end{cases}
\end{equation*}
For $i\in\{1,\dots, 4\},$ we let $\Omega_{Y,i}$ be the subset of  $i_Y(\Delta^c(\mathcal T,\epsilon_i))$ consisting of the elements coming from $\Omega_1^{s+}(\Sigma_{0,4}),$ and let $\Omega_Y=\coprod_{i=1}^4\Omega_{Y,i}.$ Then we have the following lemma whose proof is similar to that of Lemma \ref{MM3}.

\begin{lemma}\label{MM5}
For every $k\in(-2,2),$ the ellipse 
$$E_{Y,k}=\{(a,1,c)\in \Omega_Y\ |\ a^2+c^2-1=kac\}$$
 is invariant under the action of  $\langle D_Y\rangle,$ and for almost every $k\in(-2,2),$ the action of $\langle D_Y\rangle$ on $E_{Y,k}$ is ergodic.
 \end{lemma}

\begin{proof} [Proof of Theorem \ref{E} (2)] A direct calculation shows that the two family of curves $\{i_X^{-1}(E_{X,k})\}$ and $\{i_Y^{-1}(E_{Y,k})\}$ transversely intersect. Then by Lemma \ref{MM3} and Lemma \ref{MM5}, the $Mod(\Sigma_{0,4})$-action on $\mathcal M_1^{s_+}(\Sigma_{0,4})$ is ergodic.
\end{proof}

\subsection{A proof of Theorem \ref{E} (3)}

\begin{proof} By symmetry, it suffices to prove the ergodicity of the $Mod(\Sigma_{0,4})$-action on $\Omega_1^{s_{1}}(\Sigma_{0,4}).$ We let $\Delta_x=\{(a,b,c)\in \Delta\ |\  a>b+c\},$ $\Delta_y=\{(a,b,c)\in\Delta\ |\ b>a+c\}$ and $\Delta_z=\{(a,b,c)\in\Delta\ |\ c>a+b\}.$ By Theorem \ref{Kashaev2},  Lemma \ref{delta} and Corollary \ref{Delta}, given a tetrahedral triangulation of $\Sigma_{0,4},$ $\Delta_x\coprod \Delta_y\coprod\Delta_z$ is diffeomorphic to a dense and open subset of $\mathcal M_1^{s_{1}}(\Sigma_{0,4}),$ where the diffeomorphism is given by the quantities $\lambda(x),$ $\lambda(y)$ and $\lambda(z).$ 

Let 
$$R=\{(s,t)\in \mathbb R^2\ |\ s\neq 0, t\neq 0, s+t\neq 0\},$$
and consider the two-fold covering map $\psi:R\rightarrow \Delta_x\coprod \Delta_y\coprod\Delta_z$ defined by $$\psi((s,t))=\big[\sinh |s|, \sinh |t|, \sinh |s+t|\big],$$
where $[a,b,c]\doteq(\frac{a}{a+b+c},\frac{b}{a+b+c},\frac{c}{a+b+c}).$
Let $\Omega$ be the subset of $\Delta_x\coprod \Delta_y\coprod\Delta_z$ consisting of the elements coming from $\Omega_1^{s_1}(\Sigma_{0,t}),$ and let $\Omega'=\psi^{-1}(\Omega).$ Then by Proposition \ref{dense}, $\Omega$ is invariant under the $Mod(\Sigma_{0,4})$-action. Recall that $Mod(\Sigma_{0,4})$ is isomorphic to a free group $F_2$ of rank two generated by the Dehn twists $D_X$ and $D_Y.$ (See e.g. \cite{FM}.) It is well known that 
$F_2$ is isomorphic to the quotient group $\Gamma(2)/\pm I,$
where 
\begin{equation*}
\Gamma(2)=\Big\{A\in SL(2,\mathbb Z)\ \Big|\ A\equiv
 \left( \begin{array}{cc}
 1 & 0 \\ 
 0&1 \\
\end{array} \right)\ (\text{mod }2)\Big\}
\end{equation*}
is the mod-$2$ congruence subgroup of $SL(2,\mathbb Z),$ and the matrices $\left[\begin{array}{cc}
1 & 0\\
2 & 1\\
\end{array}\right]$ and $\left[\begin{array}{cc}
1 & 2\\
0 & 1\\
\end{array}\right]$ correspond to the generators of $F_2.$ This induces a group homomorphism $\pi:\Gamma(2)\rightarrow Mod(\Sigma_{0,4})$ defined by 
\begin{equation*}
\pi\Big(\left[\begin{array}{cc}
1 & 0\\
2 & 1\\
\end{array}\right]\Big)= D_X \quad \text{and}\quad \pi\Big(\left[\begin{array}{cc}
1 & 2\\
0 & 1\\
\end{array}\right]\Big)= D_Y.
\end{equation*}
By Lemma \ref{lambda'} and a direct calculation, $\psi:\Omega'\rightarrow \Omega$ is $\pi$-equivariant, i.e. $\psi\circ A=\pi(A)\circ \psi$ for all $A\in \Gamma (2),$ where the $\Gamma(2)$-action on $\Omega'$ is the standard linear action. By Moore's Ergodicity Theorem\,\cite{M}, the $\Gamma(2)$-action on $\mathbb R^2,$  hence on $\Omega',$ is ergodic. Therefore, the $Mod(\Sigma_{0,4})$-action on $\Omega$ is ergodic. \end{proof}


\appendix
\section{Equivalence of extremal and Fuchsian representations}\label{A}

Proposition \ref{MF} below is a consequence of Goldman (\cite{G2}, Theorem D) and is stated without proof in \cite{Bo}. The purpose of this appendix is to include a proof of it for the readers' interest, where the argument is from a discussion with F. Palesi and M. Wolff.

\begin{proposition}\label{MF}
A type-preserving representation $\rho:\pi_1(\Sigma_{g,n})\rightarrow PSL(2,\mathbb R)$ is extremal, i.e., $|e(\rho)|=2g-2+n,$  if and only if $\rho$ is Fuchsian.
\end{proposition}

\begin{proof} By Goldman\,\cite{G2}, Theorem D, a representation $\rho$ is maximum if and only if $\rho$ is Fuchsian and the quotient $\mathbb H^3/ \rho(\pi_1(\Sigma_{g,n}))$ is homeomorphic to $\Sigma_{g,n}.$ Therefore, to prove the Proposition, it suffices to rule out the possibility that  $\rho$ is non-maximum, Fuchsian and $\mathbb H^3/ \rho(\pi_1(\Sigma_{g,n}))=\Sigma_{g',n'}\ncong \Sigma_{g,,n,},$  which we will do using a contradiction. 

Now since $\mathbb H^3/ \rho(\pi_1(\Sigma_{g,n}))=\Sigma_{g',n'},$ there is an isomorphism
 $$\phi:\pi_1(\Sigma_{g',n'})\rightarrow \rho(\pi_1(\Sigma_{g,n}));$$
and since $\rho$ is type-preserving, $\phi(\pi_1(\Sigma_{g',n'}))= \rho(\pi_1(\Sigma_{g,n}))$ contains $n$  parabolic elements from the primitive peripheral elements of $\pi_1(\Sigma_{g,n}).$ On the other hand, since the only possible parabolic elements of a Fuchsian subgroup of $PSL(2,\mathbb R)$ are from the peripheral elements, the composition
$$\phi^{-1}\circ\rho:\pi_1(\Sigma_{g,n})\rightarrow\pi_1(\Sigma_{g',n'})$$
must send the primitive peripheral elements of $\pi_1(\Sigma_{g,n})$ to the primitive peripheral elements of $\pi_1(\Sigma_{g',n'}).$ This is impossible when $n>n',$ since $\pi_1(\Sigma_{g',n'})$ has only $n'$ primitive peripheral elements. For the case $n<n',$ we recall the fact that in the first homology $H_1(\Sigma,\mathbb R)$ of a punctured surface $\Sigma,$ the full set of vectors represented by the primitive peripheral elements of $\pi_1(\Sigma)$ are linearly dependent, but the vectors in any proper subset of it are linearly independent. Therefore, the induced isomorphism 
$$(\phi^{-1}\circ\rho)^*:H_1(\Sigma_{g,n};\mathbb R)\rightarrow H_1(\Sigma_{g',n'};\mathbb R)$$
sends a set of linearly dependent vectors represented by the primitive peripheral elements of $\pi_1(\Sigma_{g,n})$ to a set of linearly indecent vectors, which is a contradiction.
\end{proof}


\section{Relationship with Goldman's work on one-punctured torus}\label{B}

In this appendix, we show that the results concerning representations of relative Euler class $\pm1$ in this paper can also be seen, and more straightforwardly, as consequences of some previous results of Goldman (\cite{G5}, Chapter 4), where the argument presented here is due to the anonymous referee. 

In \cite{G5}, Goldman considers $SL(2,\mathbb R)$-representations of the one-puncture torus group $\pi(\Sigma_{1,1}),$ which is the free group of two generators $\langle X,Y\rangle.$ The character space $\mathcal M_{\text{red}}(\Sigma_{1,1})$ of reducible representations $\rho:\pi(\Sigma_{1,1})\rightarrow SL(2,\mathbb R)$ satisfy $tr(\rho[X,Y])=2,$ and hence could be described by the equation 
\begin{equation}\label{11}
x^2+y^2+z^2-xyz-4=0
\end{equation}
where $x=tr(\rho(X)), y=tr(\rho(Y))$ and $z=tr(\rho(XY)).$ On the other hand, the fundamental group of the four puncture sphere 
$\pi_1(\Sigma_{0,4})\cong\langle A,B,C,D\ |\ ABCD\rangle,$
where the generators are the four primitive peripheral elements corresponding to the four punctures. If $\rho:\pi_1(\Sigma_{0,4})\rightarrow PSL(2,\mathbb R)$ is type-preserving, then $|tr(\rho(A))|=|tr(\rho(B))|=|tr(\rho(C))|=|tr(\rho(D))|=2;$ and if $e(\rho)=\pm 1,$ then one can lift $\rho$ to a representation 
$\widetilde{\rho}:\pi_1(\Sigma_{0,4})\rightarrow SL(2,\mathbb R)$ such that $tr(\widetilde\rho(A))tr(\widetilde\rho(B))tr(\widetilde\rho(C))tr(\widetilde\rho(D))<0.$ Hence the character spaces $\mathcal M_{\pm 1}(\Sigma_{0,4})$ can be described by the equation 
\begin{equation}\label{04}
x^2+y^2+z^2+xyz-4=0
\end{equation}
where $x=tr(\rho(AB)), y=tr(\rho(BC))$ and $z=tr(\rho(CA)).$ (See e.g. \cite{G,MPT} for more details.) Comparing (\ref{11}) and (\ref{04}), it is clear that 
$$\mathcal M_{\text{red}}(\Sigma_{1,1})\cong\mathcal M_{\pm 1}(\Sigma_{0,4}).$$
Moreover, the mapping class group actions are commensurable and the variables $x,y,z$ correspond in each case to the traces of simple closed curves on the surface, hence all the results known for $\mathcal M_{\text{red}}(\Sigma_{1,1})$ can be translated to the results on $\mathcal M_{\pm 1}(\Sigma_{0,4}).$

To be more precise, by (\cite{G5}, Chapter 4), $\mathcal M_{\text{red}}(\Sigma_{1,1})$ has five connected components, one of which is compact corresponding to $\mathcal M_{\pm 1}^s(\Sigma_{0,4})$ and four of which are non-compact corresponding to $\mathcal M_{\pm 1}^{s_i}(\Sigma_{0,4}).$ A full measure subset of the characters in the non-compact components have all coordinates $x,y,z$ strictly greater than $2$ in absolute value. Each coordinate correspond to the trace of the image of a simple closed curve.  Starting from a representation in one of these components and using the transitivity of the mapping class group action on the set of simple closed curves, one gets that every simple closed curve is sent to an hyperbolic element. Therefore, a full measure subset of representations in the non-compact components are counterexamples to Bowditch's question. Finally, the ergodicity of the $PSL(2,\mathbb Z)$-action on the non-compact components is already proved in (\cite{G5}, Chapter 4), implying the $Mod(\Sigma_{0,4})$-action on $\mathcal M_{\pm 1}^{s_i}(\Sigma_{0,4}).$


\noindent
Tian Yang\\
Department of Mathematics, Stanford University\\
Stanford, CA 94305, USA\\
(yangtian@math.stanford.edu)


\begin{thebibliography}{99}


\bibitem{B2} Bonahon, F.,  {\em Shearing hyperbolic surfaces, bending pleated surfaces and Thurston's symplectic fortm}, Ann. Fac. Sci. Toulouse Math. 5 (1996),  233--297.

\bibitem{B} --------,  {\em Low-dimensional geometry: From Euclidean surfaces to hyperbolic knots}, Student Math. Library, IAS/Park City Mathematical Subseries 49, Amer. Math. Soc. (2009).


\bibitem{BW} Bonahon, F. and Wong, H., {\em Quantum traces for representations of surface groups in $SL_2(\mathbb C)$}, Geom. Topol. 15 (2011), no. 3, 1569--1615.

\bibitem{Bo} Bowditch, B. H., {\em Markoff triples and quasi-Fuchsian groups}, Proc. London Math. Soc. (3) 77 (1998), no. 3, 697--736.

\bibitem{D} Delgado, R., {\em Density properties of Euler characteristic-$2$ surface group, $PSL(2,\mathbb R)$ character varieties}, Thesis (Ph.D.)--University of Maryland, College Park. 2009.

\bibitem {DT} Deroin, B. and Tholozan. N., {\em Dominating surface group representations by
fuchsian ones}, preprint.

\bibitem{FM} Farb, B. and Margalit, D., {\em A primer on mapping class groups}, Princeton Mathematical Series, 49. Princeton University Press, Princeton, NJ, 2012.

\bibitem{Fr} Fricke, R. and Klein, F., {\em Vorlesungen unber der automorphen funktionen}, vol. I (1987).

\bibitem{G1} Goldman, W. M., {\em Discontinuous groups and the Euler class}, Doctoral dissertation, University of California, Berkeley 1980.


\bibitem{G4} --------, {\em The symplectic nature of fundamental groups of surfaces}, Adv. in Math. 54 (1984), no. 2, 200--225.

\bibitem{G2} --------, {\em Topological components of spaces of representations}, Invent. Math. 93 (1988), no. 3, 557--607.

\bibitem{G}  --------, {\em Ergodic theory on moduli spaces}, Ann. of Math. 146 (1997), no. 3, 475--507.

\bibitem{G5} --------, {\em The modular group action on real $SL(2)$-characters of a one-holed torus}, Geom. Topol. 7 (2003), no. 1, 443--486.

\bibitem{G3} --------, {\em Mapping class group dynamics on surface group representations}, Problems on mapping class groups and related topics, 189--214, Proc. Sympos. Pure Math., 74, Amer. Math. Soc., Providence, RI, 2006.

\bibitem{GX} Goldman, W. M. and Xia, E, {\em Ergodicity of mapping class group actions on SU(2)-character varieties}, Geometry, rigidity, and group actions, 591--608, Chicago Lectures in Math., Univ. Chicago Press, Chicago, IL, 2011.

\bibitem{GKW} Gu\'eritaud, F., Kassel, F. and Wolff, M., {\em Compact anti-de Sitter 3-manifolds and folded hyperbolic structures on surfaces}, preprint.

\bibitem{Ha} Harer, J., {\em The virtual cohomological dimension of the mapping class group of an oriented surface}, Invent. Math. 84 (1986), 157--176.

\bibitem{Ka1} Kashaev, R. M., {\em Coordinates for the moduli space of flat $PSL(2,\mathbb R)$-connections}, Math. Res. Lett. 12 (2005), no. 1, 23--36.

\bibitem{Ka2} --------, {\em On quantum moduli space of flat $PSL_2(\mathbb R)$-connections on a punctured surface}, Handbook of Teichm\"uller theory. Vol. I, 761--782, IRMA Lect. Math. Theor. Phys., 11, Eur. Math. Soc., Z\"urich, 2007.

\bibitem{KS} Keen, L. and Series C., {\em The Riley slice of Schottky space}, Proc. London Math. Soc. (3) 69 (1994), no. 1, 72--90.

\bibitem{M} Moore, C.C., {\em Ergodicity of flows on homogeneous spaces}, Amer. J. Math. 88
(1966) 154--178.

\bibitem{MPT} Maloni, S., Palesi, F. and Tan, S. P., {\em On the character varieties of the four-holed sphere}, to appear in Groups, Geometry, and Dynamics.

\bibitem{MW} March\'e, J. and Wolff, M., {\em The modular action on $PSL(2,\mathbb R)$-characters in genus 2}, preprint.

\bibitem{P1} Penner, R., {\em The decorated Teichm\"uller space of punctured surfaces}, Commun. Math. Phys. 113 (1987) 299--339.

\bibitem{P2}  --------, {\em Weil-Pertersson volumes}, J. Differential Geom 35 (1992), 559--608.

\bibitem{RY} Roger, J. and Yang, T., {\em  The skein algebra of arcs and links and the decorated Teichm\"uller space}, J. Differential Geom. 96 (2014), no.1, 95--140.

\bibitem{S} Souto, J., private communication.

\bibitem{SY} Sun, Z and Yang, T., private communication.

\bibitem{TWZ} Tan, S., Wong Y. and Zhang, Y., {\em Generalized Markoff maps and McShane's identity}, Adv. in Math.  217  (2008),  no. 2, 761--813.

\end{thebibliography}
\end{document}